\documentclass{amsart}
\usepackage{mythmsub}
\usepackage{tikz-cd}
\usepackage[shortlabels]{enumitem}
\usepackage[bookmarks]{hyperref}
\usepackage{amsthm}
\numberwithin{equation}{thm}
\newcommand{\dmeff}{{\rm DM}_{gm}^{e\! f\! f}}
\newcommand{\mmeff}{{\rm MM}^{e\! f\! f}}
\newcommand{\eff}{{e\! f\! f}}
\newcommand{\dm}{{\rm DM}_{gm}}
\begin{document}
\title{Description of mixed motives}
\author{Doosung Park}
\address{Institut f\"ur Mathematik, Universit\"at Z\"urich, Winterthurerstr. 190, 8057 Z\"urich, Switzerland}
\email{doosung,park@math.uzh.ch}
\begin{abstract}
  Assuming the K\"unneth type standard conjecture, we propose a way to describe objects of mixed motives explicitly. We study their formal properties, and we associate mixed motives to schemes smooth and separated over a field. This serves as a universal cohomology theory. We also discuss $\ell$-adic realizations, and we discuss an unconditional construction of 2-motives and their properties.
\end{abstract}
\maketitle
\section{Introduction}
\begin{none}\label{1.2}
Throughout this paper, fix a perfect field $k$ with an embedding $k\hookrightarrow \overline{k}$ to its algebraic closure.

A conjecture (\cite[Definition 2.20]{Jan}) is that there is an abelian category of mixed motives ${\rm MM}$ over $k$ such that for each mixed motive $M$, there is a weight filtration
\begin{equation}
\label{1.2.1}
M=W^{r}M\rightarrow \cdots \rightarrow W^{-1}M=0
\end{equation}
such that $W^{i}M/W^{i-1}M$ is a pure motive of weight $i$ for each $i$.

Voevodsky constructed the triangulated category ${\rm DM}_{gm}^\eff(k,{\bf Z})$ of geometric effective motives and the triangulated category ${\rm DM}_{gm}(k,{\bf Z})$ of geometric motives (\cite{Voe2}).
If there is a conic over $k$ with no rational points, then ${\rm DM}_{gm}(k,{\bf Z})$ cannot have a reasonable $t$-structure whose heart is ${\rm MM}$ (\cite[Proposition 4.3.8]{Voe2}).
It is still conjectured that the triangulated category ${\rm DM}_{gm}(k,{\bf Q})$ of motives with ${\bf Q}$-coefficient has a reasonable $t$-structure whose heart is ${\rm MM}$.
Indeed, in \cite{Han} there is a conditional proof of the existence of a reasonable $t$-structure of ${\rm DM}_{gm}(k,{\bf Q})$ assuming several conjectures including the Hanamura vanishing conjecture given as follows.
    \item[(${\rm Van}_{\leq n}$)] Assume $({\rm C}_{\leq n})$. For any $X,Y\in SmProj_{\leq n}$ and $d,e\geq 0$,
    \[{\rm Hom}(M_d(X),M_e(Y)[p])=0\]
Here,  we denote by $SmProj_{\leq n}$ the category of schemes smooth and projective over $k$ whose dimensions are $\leq n$.

This is one description of the category ${\rm MM}$ assuming the existence of a reasonable $t$-structure.
How can we describe objects of ${\rm MM}$ explicitly?
This question is what we want to study in this paper.

For simplicity of notations, we set
  \[\dmeff:={\rm DM}_{gm}^\eff(k,{\bf Q}),\quad \dm:=\dm(k,{\bf Q}), \quad {\bf L}={\bf 1}(1)[2].\]
  Here, ${\bf 1}$ denotes the object $M(k)$ of $\dmeff$.
\end{none}

\begin{none}\label{1.15}
Since ${\rm MM}$ should be an extension of the category of numerical motives (\cite{Jan3}), we may need to assume the K\"unneth type standard conjecture (\cite{Gro}).
Even if $k$ is not algebraically closed field, we can still formulate this conjecture as follows.
  \begin{enumerate}
    \item[$({\rm C}_{\leq n})$] For any $X\in  SmProj_{\leq n}$ and $d\geq 0$, there is a projector $p_d:M(X)\rightarrow M(X)$ such that the corresponding homomorphism
\[
H_{\et}^*(X_{\overline{k}},{\bf Q}_{\ell})\rightarrow H_{\et}^*(X_{\overline{k}},{\bf Q}_{\ell})
\]
is the composition
        \[H_{\et}^*(X_{\overline{k}},{\bf Q}_\ell)\rightarrow H_{\et}^d(X_{\overline{k}},{\bf Q}_\ell)\rightarrow H_{\et}^*(X_{\overline{k}},{\bf Q}_\ell).\]
        Here, $X_{\overline{k}}:=X\times_k \overline{k}$, and the first (resp.\ second) arrow is the obvious projection (resp.\ inclusion). We denote by $M_d(X)$ the image of $p_d$ in $\dmeff$, which exists since $\dmeff$ is a pseudo-abelian category.
\end{enumerate}
Then we can provide a description of ${\rm MM}$ inspired by the weight filtration \eqref{1.2.1} as follows.
\end{none}

\begin{df}\label{1.3}
Assume $({\rm C}_{\leq n})$.
  \begin{enumerate}[(1)]
    \item We denote by
    \[Gr_d{\rm MM}_{\leq n}\]
    the full subcategory of $\dmeff$ consisting of elements of the form $M[-d]$ where $M$ is a direct summand of $M_d(X)$ for some $X\in SmProj_{\leq n}$. An object of $Gr_d{\rm MM}_{\leq n}$ is called a {\it pure $n$-motive} of weight $d$.
    \item Set $W^{d}{\rm MM}_{\leq n}:=0$ for $d<0$. For $d\geq 0$, we inductively denote by
    \[W^d{\rm MM}_{\leq n}\]
     the full subcategory of $\dmeff$ consisting of objects $M$ such that there is a distinguished triangle
        \[M'\rightarrow M\rightarrow M''\rightarrow M'[1]\]
        in $\dmeff$ with $M'\in {\rm ob}\,Gr_d{\rm MM}_{\leq n}$ and $M''\in {\rm ob}\, W^{d-1}{\rm MM}_{\leq n}$. An object of $W^d{\rm MM}_{\leq n}$ is called an {\it $n$-motive of weights $\leq d$}.\
    \item We denote by
    \[{\rm MM}_{\leq n}\]
    the union of $W^d{\rm MM}_{\leq n}$ for $d\geq 0$. An object of ${\rm MM}_{\leq n}$ is called an {\it $n$-motive}.
  \end{enumerate}
\end{df}
\begin{df}\label{1.4}
  Assume $({\rm C}_{\leq n})$ for any $n$.
  \begin{enumerate}[(1)]
    \item We denote by
    \[W^d\mmeff\]
    the union of $W^d{\rm MM}_{\leq n}$ for $n\geq 0$. An object of $W^d\mmeff$ is called an {\it effective mixed motive of weights $\leq d$.}
    \item We denote by
    \[\mmeff\]
    the union of ${\rm MM}_{\leq n}$ for $n\geq 0$. An object of $\mmeff$ is called an {\it effective mixed motive}.
    \item We denote by
    \[{\rm MM}\]
     the full subcategory of $\dm$ consisting of objects of the form $M(r)$ for $M\in {\rm ob} \,\mmeff$ and $r\in {\bf Z}$. An object of ${\rm MM}$ is called a {\it mixed motive}.
  \end{enumerate}
  If one needs to emphasize that our coefficient ring is ${\bf Q}$, one may call a mixed motive as {\it mixed motive with ${\bf Q}$-coefficient}.
\end{df}

\begin{none}\label{1.5}
  Then we study formal properties of mixed motives. We first study the weight filtration functor $W^d:{\rm MM}_{\leq n}\rightarrow W^d{\rm MM}_{\leq n}$.
For this purpose, we may need to assume the Murre vanishing conjecture (\cite[Proposition 5.8]{Jan2}), which is as follows.  
\begin{enumerate}
    \item[(${\rm Mur}_{\leq n}$)] Assume $({\rm C}_{\leq n})$. For any $X,Y\in SmProj_{\leq n}$ and $d>e\geq 0$,
    \[{\rm Hom}(M_d(X),M_e(Y))=0.\]
\end{enumerate}
Then we construct the functor $W^d$ as follows and provide a definition of $n$-motives of weights $\geq d$.
\end{none}

\begin{thm}[Theorem \ref{3.7}]\label{1.6}
  Assume $({\rm C}_{\leq n})$ and $({\rm Mur}_{\leq n})$. Then the inclusion \[W^d{\rm MM}_{\leq n}\rightarrow {\rm MM}_{\leq n}\] admits a left adjoint denoted by $W^d:{\rm MM}_{\leq n}\rightarrow W^d{\rm MM}_{\leq n}$.
\end{thm}

\begin{df}\label{1.16}
Assume $({\rm C}_{\leq n})$ and $({\rm Mur}_{\leq n})$. We denote by
\[
W_d{\rm MM}_\leq n
\]
the full subcategory of ${\rm MM}_{\leq n}$ consisting of objects $M$ such that $W^{d-1}M=0$. Its object is called an $n$-motive of weights $\geq d$. If $M$ is in the intersection of $W_d{\rm MM}_\leq n$ and $W^e{\rm MM}_\leq n$, we say that $M$ is an $n$-motive of weights in $[d,e]$.
\end{df}

\begin{none}
 We also study the duality for mixed motives, which is the generalization of Cartier duality for 1-motives (\cite[1.5]{BV-S}).
\end{none}

\begin{thm}[Theorem \ref{4.2}]\label{1.7}
 Assume $({\rm C}_{\leq n})$ and $({\rm Mur}_{\leq n})$. Let $M$ be an $n$-motive of weights $\leq d$. Then $\underline{\rm Hom}(M,{\bf 1}(n))$ is an $n$-motive of weights $\geq {2n-d}$. Here, $\underline{\rm Hom}$ denotes the internal hom of $\dmeff$.
\end{thm}

\begin{none}
An object $M$ of ${\rm MM}_{\leq n}$ can be explicitly described as a commutative diagram
\[
\begin{tikzcd}
Gr_{2n}M\arrow[d]\arrow[rd,leftarrow,"{[1]}"]&
Gr_{2n-1}M\arrow[d]\arrow[rd,leftarrow,"{[1]}"]&
Gr_{2n-2}M\arrow[d]\arrow[rd,leftarrow,"{[1]}"]&
\,\arrow[rd,leftarrow,"{[1]}"]&
Gr_0M\arrow[rd,leftarrow,"{[1]}"]\arrow[d]
\\
M\arrow[r]&
W^{2n-1}M\arrow[r]&
W^{2n-2}M\arrow[r]&
\cdots\arrow[r]&
W^0M\arrow[r]&
0
\end{tikzcd}
\]
such that each triangle is a distinguished triangle and $Gr_dM$ is a pure $n$-motive of weight $d$ for any $d$.
Alternatively, $M$ can be described as a commutative diagram
\[
\begin{tikzcd}
0\arrow[r]\arrow[rd,leftarrow,"{[1]}"']&
W_{2n}M\arrow[r]\arrow[d]\arrow[rd,leftarrow,"{[1]}"']&
\cdots\arrow[r]\arrow[rd,leftarrow,"{[1]}"']&
W_2M\arrow[r]\arrow[d]\arrow[rd,leftarrow,"{[1]}"']&
W_1M\arrow[r]\arrow[d]\arrow[rd,leftarrow,"{[1]}"']&
M\arrow[d]
\\
&
Gr_{2n}M&
\,&
Gr_2M&
Gr_1M&
Gr_0M
\end{tikzcd}
\]
such that each triangle is a distinguished triangle.

In particular, if $n=1$, then $M$ is precisely a Deligne $1$-motive $[Gr_0M\rightarrow W_1M[1]]$ with a diagram
\[
\begin{tikzcd}
&
&
Gr_0M\arrow[d]
\\
0\arrow[r]&
Gr_{2}M[1]\arrow[r]&
W_1M[1]\arrow[r]&
Gr_1M[1]\arrow[r]&
0
\end{tikzcd}
\]
such that
\begin{enumerate}
\item[(i)] $Gr_0M$ is a pure $0$-motive (a.k.a. lattice),
\item[(ii)] $Gr_1M[1]$ is an abelian variety,
\item[(iii)] $Gr_2M[1]$ is a torus,
\item[(iv)] $W_1M[1]$ is a semi-abelian variety,
\item[(v)] the bottom row is an exact sequence.
\end{enumerate}
Thus our description is a generalization of Deligne $1$-motives.
\end{none}

\begin{none}
The category ${\rm MM}$ should be an abelian category containing the category of numerical motives.
We expect that the category of numerical motives is equivalent to the full subcategory of $\dm$ consisting of finite direct sums
\[
M_1\oplus \cdots \oplus M_r
\]
such that each $M_i$ is a pure $n_i$-motive of weight $d_i$ for some $n_i$ and $d_i$.
the category of numerical motives is semisimple by \cite{MR1150598}.
Hence we may need to assume additional conjectures to ensure that the above category is semisimple, which are as follows.
\begin{enumerate}
    \item[(${\rm wVan}_{\leq n}$)] Assume $({\rm C}_{\leq n})$. For any $X,Y\in SmProj_{\leq n}$ and $e>d\geq 0$,
    \[{\rm Hom}(M_d(X)[-d],M_e(Y)[-e])=0.\]
    \item[(${\rm Semi}_{\leq n}$)] Assume $({\rm C}_{\leq n})$. For any $X\in SmProj_{\leq n}$ and $d\geq 0$,
    \[{\rm Hom}(M_d(X),M_d(X))\]
    is a finite dimensional semisimple ring.
\end{enumerate}
Then we have the following result.
\end{none}
\begin{thm}[Theorem \ref{5.2}]\label{1.8}
Assume $({\rm C}_{\leq n})$, $({\rm Mur}_{\leq n})$, $({\rm wVan}_{\leq n})$, and $({\rm Semi}_{\leq n})$.
Then ${\rm MM}_{\leq n}$ is an abelian category.
\end{thm}

\begin{none}
In \cite{MR1225267}, the following conjecture is introduced.
\begin{enumerate}
    \item[$({\rm CK}_{\leq n})$] For any $X\in  SmProj_{\leq n}$, there are projectors $p_0,\ldots,p_{2n}:M(X)\rightarrow M(X)$ satisfying the condition in $({\rm C}_{\leq n})$ such that
    \[p_i\circ p_j=0\]
    if $i\neq j$.
\end{enumerate}
 If $({\rm CK}_{\leq n})$ holds, then for any $X\in SmProj_{\leq n}$, there is a decomposition
  \[M(X)=M_0(X)\oplus \cdots \oplus M_{2n}(X)\]
  in $\dmeff$,
which is called the {\it Chow-K\"unneth} decomposition.
This decomposition is {\it not} functorial.
Set
  \[M_{\leq d}(X):=M_0(X)\oplus \cdots \oplus M_d(X)\]
  for each $d$.
We show that the morphism $M(X)\rightarrow M_{\leq d}(X)$ induced by the decomposition is functorial assuming $({\rm CK}_{\leq n})$, $({\rm Mur}_{\leq n})$, and the following conjecture.
\begin{enumerate}
    \item[(${\rm wVan}_{\leq n}'$)] Assume $({\rm C}_{\leq n})$. For any $X,Y\in SmProj_{\leq n}$ and $d>e\geq 0$,
    \[{\rm Hom}(M_d(X),M_e(Y)[-1])=0.\]
\end{enumerate}
\end{none}

\begin{none}\label{1.9}
  One of the motivations for mixed motives is to construct a sort of universal cohomology for schemes not necessarily projective over $k$. Assuming $({\rm C}_{\leq n})$, we construct the $\ell$-adic realization functor
  \[R_\ell:{\rm MM}_{\leq n}\rightarrow {\rm Rep}_{{\rm Gal}(\overline{k}/k)}({\bf Q}_\ell).\]
  Here, ${\rm Rep}_{{\rm Gal}(\overline{k}/k)}({\bf Q}_\ell)$ denotes the category of ${\bf Q}_\ell$-representations of ${\rm Gal}(\overline{k}/k)$.
  
Let $U$ be a scheme smooth and separated over $k$ whose dimension is $\leq n$.
To study the $\ell$-adic cohomology using ${\rm MM}_{\leq n}$, we need to associate $n$-motives
\begin{equation}
\label{1.9.1}
M_0(U),\ldots,M_{2n}(U),M_0^c(U),\ldots,M_{2n}^c(U).
\end{equation}
Let us introduce a weaker version of $({\rm Semi}_{\leq n})$, which is as follows.
\begin{enumerate}
    \item[$({\rm Semi}_{\leq n}')$] Assume $({\rm Semi}_{\leq n-1})$. For any $d$ and morphism $f:M\rightarrow N$ in ${\rm MM}_{\leq n}$ such that
    \begin{enumerate}[(i)]
      \item ($d<2n-2$ and $M$ is an $(n-1)$-motive of weights $\leq d$) or ($d=2n-2$ and $M$ is a pure $(n-1)$-motive of weight $d$),
      \item $N$ is a pure $n$-motive of weight $d$,
    \end{enumerate}
    the kernel, cokernel, and image of $f$ exist, and the cokernel of $f$ is a pure $n$-motive of weight $d$.
\end{enumerate}
We also use the following conjectures.
\begin{enumerate}
    \item[(${\rm Res}_{\leq n}$)] Resolution of singularities holds for any integral scheme separated over $k$ whose dimension is $\leq n$.
\end{enumerate}
Assuming several conjectures, we can construct the motives \eqref{1.9.1} and show that ${\rm MM}_{\leq n}$ can be considered a universal cohomology theory.
For any  purely $n$-dimensional scheme $X$ smooth over $k$ and $d\in {\bf Z}$, we adopt the \'etale homology notation
\[
H_d^{\et}(X_{\overline{k}},{\bf Q}_\ell):=H_{\et,c}^{2n-d}(X_{\overline{k}},{\bf Q}_{\ell})(n),\quad H_{d,c}^{\et}(X_{\overline{k}},{\bf Q}_\ell):=H_{\et}^{2n-d}(X_{\overline{k}},{\bf Q}_{\ell})(n).
\]
A precise statement is as follows.
\end{none}
\begin{thm}[Theorem \ref{7.7}]\label{1.10}
  Assume $({\rm CK}_{\leq n})$, $({\rm Mur}_{\leq n})$, $({\rm wVan}_{\leq n}')$, $({\rm Res}_{\leq n})$, and $({\rm Semi}_{\leq n}')$. Then for any integral scheme $U$ smooth over $k$ whose dimension is $\leq n$,
  \[R_\ell(M_d(U)[-d])\cong H_d^{\et}(U_{\overline{k}},{\bf Q}_{\ell}),\quad R_\ell(M_d^c(U)[-d])\cong H_{d,c}^{\et}(U_{\overline{k}},{\bf Q}_{\ell}).\]
\end{thm}
\begin{none}\label{1.11}
If $n=2$, then $({\rm CK}_{\leq 2})$ is known (\cite{Mur}). Thus we have an unconditional construction of ${\rm MM}_{\leq 2}$. Moreover, $({\rm Mur}_{\leq 2})$ is known (\cite[Theorem 7.3.10]{KMP}), $({\rm Res}_{\leq 2})$ is known, and we prove $({\rm Semi}_{\leq 2}')$ and $({\rm wVan}_{\leq 2}')$. Thus we get the following result
\end{none}
\begin{thm}[Theorem \ref{8.6}]\label{1.13}
  When $n=2$, the conclusions of {\rm Theorems \ref{1.6}, \ref{1.7}, and \ref{1.10}} hold.
\end{thm}

\begin{none}\label{1.14}
  Thus we have an unconditional construction of ${\rm MM}_{\leq 2}$ and its several properties. However, since $({\rm wVan}_{\leq 2})$ and $({\rm Semi}_{\leq 2})$ are open, in Theorem \ref{1.8}, we have only a conditional proof that ${\rm MM}_{\leq 2}$ is an abelian category.
  
  In \cite{Ayo}, Ayoub introduced $2$-motives using the $2$-motivic $t$-structure. The advantage of his definition is that the category of Ayoub $2$-motives is abelian. However, it requires at least some vanishing conjectures to show that $M_2(X)[-2]$, $M_3(X)[-3]$, and $M_4(X)[-4]$ are Ayoub $2$-motives where $X$ is a surface smooth and projective over $k$. Our conjecture is that the category of our $2$-motives is a full subcategory of the category of Ayoub $2$-motives.
  
For general $n$, even though the conclusion of Theorem \ref{1.10} depends on several conjectures, there is a chance that for a specific $U$, the motives \eqref{1.9.1} are easy to construct.
Then the argument in Theorem \ref{1.10} can be used to such a case.
\end{none}
\begin{none}\label{1.12}
  {\it Organizations.} In Section 2, we define the weight filtration, and we study its functoriality.
  In Section 3, we study the dual $\underline{\rm Hom}(M,{\bf 1}(n))$ of an $n$-motive $M$.
  In Section 4, we provide a conditional proof that ${\rm MM}_{\leq n}$ is an abelian category.
  In Section 5, we construct motives $M_0(U),\ldots,M_{2m}(U)$ and $M_0^c(U),\ldots,M_{2m}(U)$ in ${\rm MM}_{\leq n}$ for any scheme $U$ smooth and separated over $k$ whose dimension is $\leq n$.
  In Section 6, we define the $\ell$-adic realization functor, and we compare $M_d(U)[-d]$ with the $\ell$-adic cohomology theory.
  In Section 7, we discuss $2$-motives.
\end{none}

\section{Weight filtration}
\begin{lemma}\label{3.2}
  Let $\mathscr{T}$ be a triangulated category, and let
  \begin{equation}\label{3.2.1}\begin{tikzcd}
    M'\arrow[r,"u"]&M\arrow[d,"f"]\arrow[r,"v"]&M''\arrow[r,"w"]&M'[1]\\
    N'\arrow[r,"u'"]&N\arrow[r,"v'"]&N''\arrow[r,"w'"]&N'[1]
  \end{tikzcd}\end{equation}
  be a diagram in $\mathscr{T}$ such that each row is a distinguished triangle. Assume that
\[
{\rm Hom}(M',N'')=0,\;{\rm Hom}(M',N''[-1])=0.
\]
Then we have the following results.
  \begin{enumerate}[{\rm (1)}]
    \item There is a unique morphism $f':M'\rightarrow N'$ making the above diagram commutative.
    \item There is a unique morphism $f'':M''\rightarrow N''$ making the above diagram commutative.
    \item There is a unique pair of morphisms $(f':M'\rightarrow N',f'':M''\rightarrow N'')$ making the above diagram into a morphism of distinguished triangles in $\dmeff$.
  \end{enumerate}
\end{lemma}
\begin{proof}
  (1) Consider the exact sequence
  \[{\rm Hom}(M',N''[-1])\rightarrow {\rm Hom}(M',N')\rightarrow {\rm Hom}(M',N)\rightarrow {\rm Hom}(M',N'')\]
  of abelian groups. Since ${\rm Hom}(M',N'')=0$ and ${\rm Hom}(M',N''[-1])=0$, the second arrow is an isomorphism. Thus there is a unique morphism $f':M'\rightarrow N'$ such that $u'f'=fu$.\\[5pt]
  (2) Consider the exact sequence
  \[{\rm Hom}(M'[1],N'')\rightarrow {\rm Hom}(M'',N'')\rightarrow {\rm Hom}(M,N'')\rightarrow {\rm Hom}(M',N'')\]
  of abeilan groups. Since ${\rm Hom}(M',N'')=0$ and ${\rm Hom}(M',N''[-1])=0$, the second arrow is an isomorphism. Thus there is a unique morphism $f'':M''\rightarrow N''$ such that $f''v=v'f$.\\[5pt]
  (3) Choose $f'$ as in (1). Since each row in \eqref{3.2.1} is a distinguished triangle, there is a morphism $f'':M''\rightarrow N''$ making \eqref{3.2.1} into a morphism of distinguished triangles in $\dmeff$. The uniqueness of $f''$ follows from (2).
\end{proof}
\begin{lemma}\label{3.14}
  Under the notations and hypotheses of {\rm Lemma {\ref{3.2}}}, if $f$ is an isomorphism, then $f'$ and $f''$ are isomorphisms.
\end{lemma}
\begin{proof}
  By Lemma \ref{3.2}, we also have a morphism
  \[\begin{tikzcd}
    N'\arrow[r,"u'"]\arrow[d,"g'"']&N\arrow[d,"f^{-1}"]\arrow[r,"v'"]&N''\arrow[d,"g''"]\arrow[r,"w'"]&N'[1]\arrow[d,"{f'[1]}"]\\
    M'\arrow[r,"u"]&M\arrow[r,"v"]&M''\arrow[r,"w"]&M'[1]
  \end{tikzcd}\]
  of distinguished triangles in $\mathscr{T}$. Consider the diagram
  \[\begin{tikzcd}
    M'\arrow[r,"u"]&M\arrow[d,"{\rm id}"]\arrow[r,"v"]&M''\arrow[r,"w"]&M'[1]\\
    M'\arrow[r,"u"]&M\arrow[r,"v"]&M''\arrow[r,"w"]&M[1]
  \end{tikzcd}\]
  in $\mathscr{T}$. Then both
\[(g'f':M'\rightarrow M',\;g''f'':M''\rightarrow M''),\;({\rm id}:M'\rightarrow M',\;{\rm id}:M''\rightarrow M'')\]
 make the above diagram into a morphism of distinguished triangles in $\mathscr{T}$. Thus by Lemma \ref{3.2}, $g'f'={\rm id}$ and $g''f''={\rm id}$. By the same argument, $f'g'={\rm id}$ and $f''g''={\rm id}$. Thus $f'$ and $f''$ are isomorphisms.
\end{proof}
\begin{lemma}\label{3.3}
  Let $\mathscr{T}$ be a triangulated category, and let
  \begin{equation}\label{3.3.1}\begin{tikzcd}
    M'\arrow[r,"u"]\arrow[d,"f'"']&M\arrow[r,"v"]&M''\arrow[r,"w"]\arrow[d,"f''"]&M'[1]\arrow[d,"{f'[1]}"]\\
    N'\arrow[r,"u'"]&N\arrow[r,"v'"]&N''\arrow[r,"w'"]&N'[1]
  \end{tikzcd}\end{equation}
  be a commutative diagram in $\mathscr{T}$. If
  \[{\rm Hom}(M',N'')=0,\quad {\rm Hom}(M',N''[-1])=0,\quad {\rm Hom}(M'',N')=0,\]
  then the above diagram can be uniquely extended to a morphism
  \[\begin{tikzcd}
    M'\arrow[r,"u"]\arrow[d,"f'"']&M\arrow[d,"f"]\arrow[r,"v"]&M''\arrow[d,"f''"]\arrow[r,"w"]&M'[1]\arrow[d,"{f'[1]}"]\\
    N'\arrow[r,"u'"]&N\arrow[r,"v'"]&N''\arrow[r,"w'"]&N'[1]
  \end{tikzcd}\]
  of distinguished triangles in $\mathscr{T}$.
\end{lemma}
\begin{proof}
  Note that such an $f$ exists since the rows are distinguished triangles. Hence it remains to show that $f$ is unique. Let $f,f_0:M\rightarrow N$ be morphisms making \eqref{3.3.1} still commutative. Then
\[
v'(f-f_0)=v'f-v'f_0=f''v-f''v=0,
\]
so $f-f_0=u'g$ for some morphism $g:M\rightarrow N'$ in $\mathscr{T}$.
Thus
  \[u'gu=(f-f_0)u=fu-f_0u=u'f'-u'f'=0.\]
  Consider the commutative diagram
  \[\begin{tikzcd}
    M'\arrow[r,"u"]\arrow[d,"gu"']&M\arrow[d,"0"]\arrow[r,"v"]&M''\arrow[d,"0"]\arrow[r,"w"]&M'[1]\arrow[d,"{gu[1]}"]\\
    N'\arrow[r,"u'"]&N\arrow[r,"v'"]&N''\arrow[r,"w'"]&N'[1]
  \end{tikzcd}\]
  in $\mathscr{T}$. By Lemma \ref{3.2}, $gu=0$. Then $g=hv$ for some morphism $h:M''\rightarrow N'$ in $\mathscr{T}$. Since ${\rm Hom}(M'',N')=0$, $h$ should be $0$. Thus $f=f_0$.
\end{proof}
\begin{prop}\label{3.15}
  Assume $(C_{\leq n})$. For any $X,Y\in SmProj_{\leq n}$, $d,e\geq 0$, and $p>0$,
  \[{\rm Hom}(M_d(X),M_e(Y)[p])=0.\]
\end{prop}
\begin{proof}
  We may assume that $Y$ is connected. Since $M_d(X)$ (resp.\ $M_e(Y)$) is a direct summand of $M(X)$ (resp.\ $M(Y)$), it suffices to show that
  \[{\rm Hom}(M(X),M(Y)[p])=0.\]
  Then it suffices to show that
  \[{\rm Hom}(M(X\times Y),{\bf 1}(r)[2r+p])=0\]
  where $r$ is the dimension of $Y$ by \cite[Theorem 16.24]{MVW} (see \cite{Kel} to remove the assumption of resolution of singularity in our case). This follows from \cite[Vanishing Theorem 19.3]{MVW}.
\end{proof}
\begin{prop}\label{3.4}
  Assume $(C_{\leq n})$ and $({\rm Mur}_{\leq n})$. Let $M$ be a pure $n$-motive of weight $d$, and let $N$ be a pure $n$-motive of weight $e$. If $d>e$, then
  \[{\rm Hom}(M,N)=0,\quad {\rm Hom}(M,N[-1])=0.\]
\end{prop}
\begin{proof}
  Choose varieties $X$ and $Y$ projective over $k$ with dimensions $\leq n$ such that $M$ (resp.\ $N$) is a direct summand of $M_d(X)[-d]$ (resp.\ $M_e(Y)[-e]$). Then it suffices to show that
  \[{\rm Hom}(M_d(X),M_e(Y)[d-e])=0,\quad {\rm Hom}(M_d(X),M_e(Y)[d-e-1])=0.\]
  The first one follows from Proposition \ref{3.15}. The second one follows from Proposition \ref{3.15} if $d>e+1$ and from $({\rm Mur}_{\leq n})$ if $d=e+1$.
\end{proof}
\begin{none}\label{3.5}
  For each $n$-motive $M$, by definition, we can choose morphisms
  \[M=W^{2n}M\rightarrow W^{2n-1}M\rightarrow \cdots \rightarrow W^{-1}M=0\]
  in $\dmeff$ such that a cone of $W^dM\rightarrow W^{d-1}M$ is a pure $n$-motive of weight $d$.
\end{none}
\begin{prop}\label{3.6}
  Assume $({\rm C}_{\leq n})$ and $({\rm Mur}_{\leq n})$. Let $M$ be a pure $n$-motive of weight $d$, and let $N$ be an $n$-motive of weights $\leq e$. If $d>e$, then
  \[{\rm Hom}(M,N)=0,\quad {\rm Hom}(M,N[-1])=0.\]
\end{prop}
\begin{proof}
  Choose morphisms $W^{2n}N\rightarrow \cdots \rightarrow W^{-1}N$ as above. By Proposition \ref{3.4}, the homomorphisms
  \[{\rm Hom}(M,W^dN)\rightarrow {\rm Hom}(M,W^{d-1}N),\]
\[
{\rm Hom}(M,W^dN[-1])\rightarrow {\rm Hom}(M,W^{d-1}N[-1])
\]
  are injective for each $d$ since a cocone of $W^dN\rightarrow W^{d-1}N$ is a pure $n$-motive of weight $d$. Thus
  \[{\rm Hom}(M,N)={\rm Hom}(M,W^{2n}N)\rightarrow {\rm Hom}(M,W^{-1}N)= 0,\]
  \[{\rm Hom}(M,N[-1])={\rm Hom}(M,W^{2n}N[-1])\rightarrow {\rm Hom}(M,W^{-1}N[-1])= 0\]
are injective, so we are done.
\end{proof}
\begin{thm}\label{3.7}
   Assume $({\rm C}_{\leq n})$ and $({\rm Mur}_{\leq n})$. The inclusion functor \[W^d {\rm MM}_{\leq n}\rightarrow {\rm MM}_{\leq n}\] admits a left adjoint denoted by $W^d:{\rm MM}_{\leq n}\rightarrow W^d {\rm MM}_{\leq n}$.
\end{thm}
\begin{none}
  We will complete the proof of Theorem \ref{3.7} in \ref{3.10}. For each $n$-motive, set $W^dM$ as in \ref{3.5}. We will construct a functor structure on $W^d$. Let $f:M\rightarrow N$ be a morphism of $n$-motives. We will show that there is a unique set of morphisms
  \[\{W^{2n-1}f:W^{2n-1}M\rightarrow W^{2n-1}N,\ldots,W^{0}f:W^0M\rightarrow W^0N\}\]
  in $\dmeff$ such that the induced diagram
  \[\begin{tikzcd}
    M=W^{2n}M\arrow[d,"f"']\arrow[r]&W^{2n-1}M\arrow[d,"W^{2n-1}"]\arrow[r]&\cdots\arrow[r]&W^dM\arrow[d,"W^df"]\\
    N=W^{2n}N\arrow[r]&W^{2n-1}N\arrow[r]&\cdots\arrow[r]&W^dN
  \end{tikzcd}\]
  of $n$-motives commutes.

  For sufficiently large $d$, $W^dM\cong M$ and $W^dN\cong N$, so the above claim holds for such a $d$. Let us use an induction on $d$. Assume that the above claim holds for $d$. Let $Gr_d M$ (resp.\ $Gr_dN$) be a cocone of $W^dM\rightarrow W^{d-1}M$ (resp.\ $W^dN\rightarrow W^{d-1}N$). Consider the induced diagram
  \[\begin{tikzcd}
    Gr_dM\arrow[r]&W^dM\arrow[d,"W^df"]\arrow[r]&W^{d-1}M\arrow[r]&Gr_dM[1]\\
    Gr_dN\arrow[r]&W^dN\arrow[r]&W^{d-1}N\arrow[r]&Gr_dN[1]
  \end{tikzcd}\]
  of $n$-motives. By Proposition \ref{3.6},
  \[{\rm Hom}(Gr_dM,W^{d-1}N)=0,\quad {\rm Hom}(Gr_dM,W^{d-1}N[-1])=0.\]
  Thus by Lemma \ref{3.2}, the above diagram can be uniquely extended to a morphism of distinguished triangles:
    \[\begin{tikzcd}
    Gr_dM\arrow[r]\arrow[d]&W^dM\arrow[d,"W^df"]\arrow[r]&W^{d-1}M\arrow[d,"W^{d-1}f"]\arrow[r]&Gr_dM[1]\arrow[d]\\
    Gr_dN\arrow[r]&W^dN\arrow[r]&W^{d-1}N\arrow[r]&Gr_dN[1]
  \end{tikzcd}\]
  Thus the above claim holds for $d-1$. This completes the induction process. Now the functoriality of $W^d$ follows from the uniqueness.
\end{none}
\begin{prop}\label{3.9}
   Assume $({\rm C}_{\leq n})$ and $({\rm Mur}_{\leq n})$. Let $M$ be an $n$-motive, and let $N$ be an $n$-motive of weights $\leq e$.
Suppose that $M$ admits a sequence of morphisms
\[M=W^{2n}M\rightarrow \cdots \rightarrow W^{d-1}M=0\]
such that for each $r\geq d$, a cocone of $W^rM\rightarrow W^{r-1}M$ is a pure $n$-motive of weight $r$.
If $d>e$, then
  \[{\rm Hom}(M,N)=0,\quad {\rm Hom}(M,N[-1])=0.\]
\end{prop}
\begin{proof}
By Proposition \ref{3.6}, the homomorphisms
  \[{\rm Hom}(W^rM,N)\rightarrow {\rm Hom}(W^{r-1}M,N)\]
    \[{\rm Hom}(W^rM,N[-1])\rightarrow {\rm Hom}(W^{r-1}M,N[-1])\]
  are injective since $r\geq d>e$.
Thus
  \[{\rm Hom}(M,N)={\rm Hom}(W^{2n}M,N)\rightarrow {\rm Hom}(W^{d-1}M,N)=0,\]
  \[{\rm Hom}(M,N[-1])={\rm Hom}(W^{2n}M,N[-1])\rightarrow {\rm Hom}(W^{d-1}M,N[-1])=0\]
are injective, so we are done.
\end{proof}
\begin{none}\label{3.10}
  {\it Proof of Theorem {\ref{3.7}}.}
  It remains to show that $W^d$ is the left adjoint of the inclusion functor. For any $n$-motive $L$ of weights $\leq d$, it suffices to show that the induced homomorphism
  \[{\rm Hom}(W^dM,L)\rightarrow {\rm Hom}(M,L)\]
  is an isomorphism. By Proposition \ref{3.9}, it suffices to show that a cocone $N$ of $M\rightarrow W^dM$ admits a sequence of morphisms
\begin{equation}
\label{3.10.1}
N=W^{2n}N\rightarrow \cdots \rightarrow W^{d}N=0
\end{equation}
such that for each $e\geq d$, a cocone of $W^{e+1}N\rightarrow W^{e}N$ is a pure $n$-motive of weight $r$.

  For $e\geq d$, let $W^eN$ be a cocone of $W^eM\rightarrow W^dM$. By the octahedral axiom, there is a commutative diagram
  \[\begin{tikzcd}
    Gr_eM\arrow[d]\arrow[r,"{\rm id}"]&Gr_eM\arrow[d]\arrow[r]&0\arrow[d]\arrow[r]&Gr_eM[1]\arrow[d]\\
    W^{e+1}N\arrow[d]\arrow[r]&W^{e+1}M\arrow[d]\arrow[r]&W^dM\arrow[d,"{\rm id}"]\arrow[r]&W^{e+1}N[1]\arrow[d]\\
    W^eN\arrow[r]\arrow[d]&W^eM\arrow[r]\arrow[d]&W^dM\arrow[r]\arrow[d]&W^eN[1]\arrow[d]\\
    Gr_eM[1]\arrow[r,"{\rm id}"]&Gr_eM[1]\arrow[r]&0\arrow[r]&Gr_eM[2]
  \end{tikzcd}\]
  in $\dmeff$ such that each row and column is a distinguished triangle. Thus we obtain \eqref{3.10.1}, and a cocone of each $W^{e+1}N\rightarrow W^eN$ is isomorphic to $Gr_eM$, which is a pure $n$-motive of weight $e$.\qed
\end{none}
Now see Definition \ref{1.16} for the definition of $W_d{\rm MM}_{\leq n}$, $n$-motives of weights $\geq d$, and $n$-motives of weights in $[d,e]$.
\begin{prop}\label{3.11}
  Assume $({\rm C}_{\leq n})$ and $({\rm Mur}_{\leq n})$.
Then the inclusion functor
\[
W_d{\rm MM}_{\leq n}\rightarrow {\rm MM}_{\leq n}
\]
admits a right adjoint denoted by
\[
W_d:{\rm MM}_{\leq n}\rightarrow W_d {\rm MM}_{\leq n}.
\]
\end{prop}
\begin{proof}
  For any $n$-motive $M$, choose a cone of $M\rightarrow W^{d-1}M$, and let us denote it by $W_dM$. We will show that $M\mapsto W_dM$ has a functor structure. Let $f:M\rightarrow N$ be a morphism of $n$-motives.  Consider the induced diagram
  \[\begin{tikzcd}
    W_dM\arrow[r]&M\arrow[d,"f"]\arrow[r]&W^{d-1}M\arrow[r]&W_dM[1]\\
    W_dN\arrow[r]&N\arrow[r]&W^{d-1}N\arrow[r]&W_dN[1]
  \end{tikzcd}\]
  of $n$-motives. By Proposition \ref{3.9},
  \[{\rm Hom}(W_dM,W^{d-1}N)=0,\quad {\rm Hom}(W_dM,W^{d-1}N[-1])=0.\]
  Thus by Lemma \ref{3.2}, the above diagram can be uniquely extended to a morphism of distinguished triangles:
    \begin{equation}\label{3.11.1}\begin{tikzcd}
    W_dM\arrow[r]\arrow[d,"W_df"']&M\arrow[d,"f"]\arrow[r]&W^{d-1}M\arrow[d,"W^{d-1}f"]\arrow[r]&W_dM[1]\arrow[d,"{W_df[1]}"]\\
    W_dN\arrow[r]&N\arrow[r]&W^{d-1}N\arrow[r]&W_dN[1]
  \end{tikzcd}\end{equation}
  The functoriality of $W_d$ follows from the uniqueness.

It remains to show that $W_d$ is the right adjoint of the inclusion functor. For any $n$-motive $L$ of weights $\geq d$, it suffices to show that the induced homomorphism
  \[{\rm Hom}(L,W_dM)\rightarrow {\rm Hom}(L,M)\]
  is an isomorphism. This follows from Proposition \ref{3.9} since a cone of $W_dM\rightarrow M$ is isomorphic to $W^{d-1}M$, which is an $n$-motive of weight $\leq d-1$.
\end{proof}
\begin{prop}\label{3.12}
  Assume $({\rm C}_{\leq n})$ and $({\rm Mur}_{\leq n})$. Consider the functors $W_d:{\rm MM}_{\leq n}\rightarrow W_d {\rm MM}_{\leq n}$ and $W^d:{\rm MM}_{\leq n}\rightarrow W^d {\rm MM}_{\leq n}$. For any $d$, there is a natural transformation $\partial:W^{d-1}\rightarrow W_d[1]$ such that
  \[W_d\rightarrow {\rm id}\rightarrow W^{d-1} \stackrel{\partial}\rightarrow W_d[1]\]
  is a distinguished triangle of functors.
\end{prop}
\begin{proof}
  For each $n$-motive $M$, choose a morphism $\partial:W^{d-1}M\rightarrow W_dM[1]$ of $n$-motives such that
  \[W_dM\rightarrow M\rightarrow W^{d-1}M\stackrel{\partial}\rightarrow W_dM[1]\]
  is a distinguished triangle.
  It remains to show that for any morphism $f:M\rightarrow N$ of $n$-motives, the diagram
  \[\begin{tikzcd}
    W^{d-1}M\arrow[d,"W^{d-1}f"']\arrow[r,"\partial"]&W_dM[1]\arrow[d,"{W_df[1]}"]\\
    W^{d-1}N\arrow[r,"\partial"]&W_dN[1]
  \end{tikzcd}\]
  of $n$-motives commutes. This follows from the commutativity of \eqref{3.11.1}.
\end{proof}
\begin{prop}\label{3.13}
  Assume $({\rm C}_{\leq n})$ and $({\rm Mur}_{\leq n})$. For any $d\geq e$, there is a natural isomorphism
  \[W^dW_e\stackrel{\sim}\rightarrow W_eW^d\]
  of functors.
\end{prop}
\begin{proof}
  Consider the commutative diagram
  \[\begin{tikzcd}
    &W_{d+1}M\arrow[d]\\
    W_eM\arrow[r]\arrow[d]&M\arrow[d]\arrow[r]&W^{e-1}M\arrow[d,"{\rm id}"]\arrow[r]&W_eM[1]\arrow[d]\\
    W_eW^dM\arrow[r]&W^dM\arrow[d]\arrow[r]&W^{e-1}M\arrow[r]&W_eW^dM[1]\\
    &W_{d+1}M[1]
  \end{tikzcd}\]
  in $\dmeff$ where the rows and the second column are distinguished triangles in Proposition \ref{3.12}. By the octahedral axiom, this can be completed to a commutative diagram
    \begin{equation}\label{3.13.1}\begin{tikzcd}
    W_{d+1}M\arrow[r,"{\rm id}"]\arrow[d]&W_{d+1}M\arrow[d]\arrow[r]&0\arrow[d]\arrow[r]&W_{d+1}M[1]\arrow[d]\\
    W_eM\arrow[r]\arrow[d]&M\arrow[d]\arrow[r]&W^{e-1}M\arrow[d,"{\rm id}"]\arrow[r]&W_eM[1]\arrow[d]\\
    W_eW^dM\arrow[r]\arrow[d]&W^dM\arrow[d]\arrow[r]&W^{e-1}M\arrow[d]\arrow[r]&W_eW^dM[1]\arrow[d]\\
    W_{d+1}M[1]\arrow[r,"{\rm id}"]&W_{d+1}M[1]\arrow[r]&0\arrow[r]&W_{d+1}M[2]
  \end{tikzcd}\end{equation}
  in $\dmeff$ where each row and column is a distinguished triangle. Consider the diagram
  \[\begin{tikzcd}
    W_{d+1}M\arrow[r]&W_eM\arrow[d,"{\rm id}"]\arrow[r]&W_eW^dM\arrow[r]&W_{d+1}M[1]\\
    W_{d+1}M\arrow[r]&W_eM\arrow[r]&W^dW_eM\arrow[r]&W_{d+1}M[1]
  \end{tikzcd}\]
  in $\dmeff$ where the first row is the first column in \eqref{3.13.1}, and the second row is induced by the distinguished triangle in Proposition \ref{3.12}. By Lemmas \ref{3.2} and \ref{3.14}, there is a unique isomorphism $W_eW^dM\stackrel{\sim}\rightarrow W^dW_eM$ making the above diagram commutative. Now the functoriality follows from the uniqueness, so this gives a natural isomorphism  $W^dW_e\stackrel{\sim}\rightarrow W_eW^d$ of functors.
\end{proof}
\begin{df}\label{3.16}
  Assume $({\rm C}_{\leq n})$ and $({\rm Mur}_{\leq n})$. For each $d$, set
  \[Gr_d:=W_dW^d.\]
\end{df}
\begin{prop}\label{3.17}
   Assume $({\rm C}_{\leq n})$, $({\rm Mur}_{\leq n})$, and $({\rm wVan}_{\leq n})$. Let $M$ be an $n$-motive of weights $\geq d$, and let $N$ be an $n$-motive of weights $\leq e$. If $d>e$, then
  \[{\rm Hom}(N,M)=0.\]
\end{prop}
\begin{proof}
  Consider the distinguished triangle
  \[Gr_eN\rightarrow N\rightarrow W^{e-1}N\rightarrow W_eN[1]\]
  in $\dmeff$. To show that ${\rm Hom}(N,M)=0$, it suffices to show that
\[
{\rm Hom}(Gr_eN,M)=0,\quad {\rm Hom}(W^{e-1}N,M)=0.
\]
Repeating this process, we reduce to the case when $N$ is a pure $n$-motive of weight $e$. By the same argument, we reduce to the case when $M$ is a pure $n$-motive of weight $d$. Then ${\rm Hom}(N,M)=0$ by $({\rm wVan}_{\leq n})$.
\end{proof}
\section{Duality}
\begin{prop}\label{4.1}
  Assume $({\rm C}_{\leq n})$ and $({\rm Mur}_{\leq n})$, and let $M$ be an object of $\dmeff$. If there is a distinguished triangle
  \[M'\stackrel{f}\rightarrow M\stackrel{g}\rightarrow M''\stackrel{h}\rightarrow M'[1]\]
  in $\dmeff$ such that $M'$ (resp.\ $M''$) is an $n$-motive of weights in $[b,c]$ (resp.\ $[a,b-1]$), then $M$ is an $n$-motive of weights in $[a,c]$.
\end{prop}
\begin{proof}
  Let us use an induction on $c$. If $c= b-1$, then we are done since $M'=0$. Assume that $c\geq b$. We have the distinguished triangle
  \[W_cM'\stackrel{u}\rightarrow M'\stackrel{v}\rightarrow W^{c-1}M'\stackrel{\partial}\rightarrow W_cM'[1]\]
  by Proposition \ref{3.12}. Then by the octahedral axiom, there is a commutative diagram
  \[\begin{tikzcd}
    W_cM'\arrow[r,"{\rm id}"]\arrow[d,"u"']&W_cM'\arrow[d]\arrow[r]&0\arrow[d]\arrow[r]&W_cM'[1]\arrow[d,"{u[1]}"]\\
    M'\arrow[r,"f"]\arrow[d,"v"']&M\arrow[d]\arrow[r,"g"]&M''\arrow[d,"{\rm id}"]\arrow[r,"h"]&M'[1]\arrow[d,"{v[1]}"]\\
    W^{c-1}M\arrow[d,"\partial"']\arrow[r]&N\arrow[d]\arrow[r]&M''\arrow[d]\arrow[r]&W^{c-1}M'[1]\arrow[d,"-{\partial[1]}"]\\
    W_cM'[1]\arrow[r,"{\rm id}"]&W_cM'[1]\arrow[r]&0\arrow[r]&W_cM'[1]
  \end{tikzcd}\]
  in $\dmeff$ such that each row and column is a distinguished triangle. Consider the third row. By induction on $c$, $N$ is an $n$-motive of weights in $[a,c-1]$. Then consider the second column. By definition, $M$ is an $n$-motive of weights $\leq c$. We also have that $W^{a-1}M=0$ since $W^{a-1}W_cM'=0$ and $W^{a-1}N=0$. Thus $M$ is an $n$-motive of weights in $[a,c]$.
\end{proof}
\begin{thm}\label{4.2}
  Assume $({\rm C}_{\leq n})$ and $({\rm Mur}_{\leq n})$. Let $M$ be an $n$-motive of weights $\leq d$. Then $\underline{\rm Hom}(M,{\bf 1}(n))$ is an $n$-motive of weights $\geq 2n-d$.
\end{thm}
\begin{proof}
  We may assume that $M$ is an $n$-motive of weights $\leq d$. Let us use an induction on $d$. If $d<0$, then we are done since $M=0$. If $d\geq 0$, we have the distinguished triangle
  \[W_dM\rightarrow M\rightarrow W^{d-1}M\rightarrow W_dM[1]\]
  in $\dmeff$. Then we get the distinguished triangle
  \[\underline{\rm Hom}(W_dM[1],{\bf 1}(n))\rightarrow \underline{\rm Hom}(W^{d-1}M,{\bf 1}(n))\rightarrow \underline{\rm Hom}(M,{\bf 1}(n))\rightarrow \underline{\rm Hom}(W_dM,{\bf 1}(n))\]
  in $\dmeff$. By induction, $\underline{\rm Hom}(W^{d-1}M,{\bf 1}(n))$ is an $n$-motive of weights $\geq 2n-d+1$. Since $W_dM$ is a pure motive of weight $d$, $\underline{\rm Hom}(W_dM,{\bf 1}(n))$ is a pure motive of weight $2n-d$. Thus by Proposition \ref{4.1}, $\underline{\rm Hom}(M,{\bf L}^n)$ is an $n$-motive of weights $\geq 2n-d$.
\end{proof}
\section{Abelian category}
\begin{lemma}\label{5.1}
  Assume $({\rm C}_{\leq n})$ and $({\rm Mur}_{\leq n})$. Let $M_1\stackrel{f}\rightarrow M_2\stackrel{g}\rightarrow M_3$ be morphisms in $\dmeff$ such that $M_1$ and $M_2$ are $n$-motives. If $Gr_df$ has a retraction for each $d$, and if $gf=0$, then $g=0$.
\end{lemma}
\begin{proof}
  We may assume that $M_1$ and $M_2$ are $n$-motives for some $n$. For $d\geq 0$, consider the induced commutative diagram
  \[\begin{tikzcd}
    Gr_dM_1\arrow[d]\arrow[r,"Gr_df"]&Gr_dM_2\arrow[d]\\
    M_1\arrow[r,"f"]\arrow[d]&M_2\arrow[r,"g"]\arrow[d]&M_3\\
    W^{d-1}M_1\arrow[r,"W^{d-1}f"]&W^{d-1}M_2
  \end{tikzcd}\]
  in $\dmeff$. The composition $Gr_dM_1\rightarrow M_3$ is $0$ since $gf=0$. Thus the composition $Gr_dM_2\rightarrow M_3$ is $0$ since $Gr_df$ has a retraction. Then $g$ factors through $W^{d-1}M$. By induction, $g$ factors through $W^{-1}M_2=0$, so $g=0$.
\end{proof}

\begin{prop}\label{2.2}
  If $({\rm Semi}_{\leq n})$ holds, then $Gr_d{\rm MM}_{\leq n}$ is a semisimple abelian category.
\end{prop}
\begin{proof}
Follows from \cite[Lemma 2]{Jan3}.
\end{proof}

\begin{thm}\label{5.2}
  Assume $({\rm C}_{\leq n})$, $({\rm Mur}_{\leq n})$, $({\rm Semi}_{\leq n})$, $({\rm wVan}_{\leq n})$, and $({\rm Semi}_{\leq n})$. Then the category ${\rm MM}_{\leq n}$ is an abelian category.
\end{thm}
\begin{proof}
  Let $f:M_1\rightarrow M_2$ be a morphism of $n$-motives, and assume that $W^dM_1\cong M_1$ and $W^dM_2\cong M_2$. Note that for any morphism of pure $n$-motives of the same weight, kernel, cokernel, and image exist by Proposition \ref{2.2}. Moreover, any kernel has a retraction, and any cokernel has a section.\\[5pt]
  (I) {\it Construction of kernels.} Consider the induced commutative diagram
  \[\begin{tikzcd}
    {\rm ker}\,W^{d-1}f\arrow[r,"h"]&W^{d-1}M_1\arrow[d,"g_1"]\arrow[r,"W^{d-1}f"]&W^{n-1}M_2\arrow[d,"g_2"]\\
    {\rm ker}\,Gr_df[1]\arrow[r,"h'"]&Gr_dM_1[1]\arrow[r,"{Gr_df[1]}"]&Gr_dM_2[1]
  \end{tikzcd}\]
  in $\dmeff$, where the upper row is given by induction on $d$. Since $h'$ is a morphism of pure $n$-motives of the same weight, $h'$ has a retraction denoted by $u'$. The composition $v:=u'g_1h$ makes the above diagram still commutative. Let ${\rm ker}\,f$ be a cocone of $v$. Then we get a morphism ${\rm ker}\,f\rightarrow M_1$ from the above diagram.

  We will show that ${\rm ker}\,f$ is the kernel of $f$. Consider a morphism $f':M_3\rightarrow M_1$ such that $ff'=0$. By induction on $d$, there is a unique morphism $w:W^{d-1}M_3\rightarrow {\rm ker}\,W^{d-1}f$ in $\dmeff$ such that $hw=W^{d-1}f'$. There is also a unique morphism $w':Gr_dM_3[1]\rightarrow {\rm ker}\,Gr_df[1]$ such that $h'w'=Gr_d f'[1]$ by Proposition \ref{2.2}. Then we have the induced commutative diagram
  \[\begin{tikzcd}
     W^{d-1}M_3\arrow[r,"w"]\arrow[d,"g_3"']&   {\rm ker}\,W^{d-1}f\arrow[r,"h"]&W^{d-1}M_1\arrow[d,"g_1"]\arrow[r,"W^{d-1}f"]&W^{d-1}M_2\arrow[d,"g_2"]\\
    Gr_dM_3[1]\arrow[r,"w'"]&{\rm ker}\,Gr_df[1]\arrow[r,"h'"]&Gr_dM_1[1]\arrow[r,"{Gr_df[1]}"]&Gr_dM_2[1]
  \end{tikzcd}\]
  in $\dmeff$. The morphism $v:{\rm ker}\,W^{d-1}f\rightarrow {\rm ker}\,Gr_d f[1]$ makes the above diagram still commutative since
  \[vw=u'g_1hw=u'g_1(hw)=u'(h'w')g_3=u'h'w'g_3=w'g_3.\]
  Consider the induced commutative diagram
  \[\begin{tikzcd}
    Gr_d M_3\arrow[d,"w'{[-1]}"']\arrow[r]&M_3\arrow[r]&W^{d-1}M_3\arrow[d,"w"]\arrow[r]&Gr_dM_3[1]\arrow[d,"{w'[1]}"]\\
    {\rm ker}\,Gr_df\arrow[r]&{\rm ker}\,f\arrow[r]&{\rm ker}\,W^{d-1}f\arrow[r,"v"]&{\rm ker}\,Gr_df[1]
  \end{tikzcd}\]
  in $\dmeff$. Then
  \[{\rm Hom}(Gr_dM_3,{\rm ker}\,W^{d-1}f)=0,\quad {\rm Hom}(Gr_dM_3,{\rm ker}\,W^{d-1}f[-1])=0,\]
  \[{\rm Hom}(W^{d-1}M_3,{\rm ker}\,Gr_df)=0\]
  by Propositions \ref{3.4} and \ref{3.17}. Thus there is a unique morphism $M_3\rightarrow {\rm ker}\,f$ in $\dmeff$ making the above diagram commutative. Now the uniqueness of kernel follows from Lemma \ref{3.3} and the fact that $w$ and $w'$ are unique.\\[5pt]
  (II) {\it Construction of cokernels.} Let $f:M_1\rightarrow M_2$ be a morphism of $n$-motives. The construction of its cokernel is similar to that of kernel, but we need to use $W_{d-1}$ instead of $W^{d-1}$. We do not repeat the proof here.\\[5pt]
  (III) {\it Final step of the proof.} It remains to show that the cokernel of the kernel agrees with the kernel of the cokernel. By induction on $d$, we suppose that this holds for $W^{d-1}f$. Consider the induced commutative diagram
  \[\begin{tikzcd}
    W^{d-1}M_1\arrow[d,"g_1"']\arrow[r,"a"]&{\rm im}\,W^{d-1}\arrow[d,"c"]\\
    Gr_dM_1\arrow[r,"a'"]&{\rm im}\,Gr_df[1]
  \end{tikzcd}\quad
  \begin{tikzcd}
    {\rm im}\,W^{d-1}\arrow[d,"c'"']\arrow[r,"b"]&W^{d-1}M_2\arrow[d,"g_2"]\\
    {\rm im}\,Gr_df[1]\arrow[r,"b'"]&Gr_dM_2[1]
  \end{tikzcd}\]
  in $\dmeff$ where $c$ (resp.\ $c'$) is obtained from the cokernel of the kernel (resp.\ the kernel of the cokernel) of $f$. We only need to show that $c=c'$.

  By Proposition \ref{2.2}, there are a retraction $a''$ of $a'$ and a section $b''$ of $b'$. Then
  \[ca=a'g_1=b''b'a'g_1=b''(b'a')g_1=b''g_2(ba)=b''g_2ba=b''b'c'a=c'a.\]
  Thus $c=c'$ by Lemma \ref{5.1}.
\end{proof}
\begin{cor}\label{5.3}
  Assume $({\rm C}_{\leq n})$, $({\rm Mur}_{\leq n})$, $({\rm Semi}_{\leq n})$, and $({\rm wVan}_{\leq n})$ for any $n$. Then the category ${\rm MM}$ is an abelian category.
\end{cor}
\begin{proof}
  For each $a\geq 0$, let $\mathcal{A}_a$ be the full subcategory of ${\rm MM}$ consisting of objects of the form $M(r)$ for $M\in {\rm ob}\,\mmeff$ and $r\geq -a$. By \cite{Voe}, the functor
  \[\dmeff\rightarrow \dmeff\]
  given by $F\mapsto F(a)$ is fully faithful. Thus $\mathcal{A}_a\simeq \mmeff$, which is an abelian category by Theorem \ref{5.2}. Then ${\rm MM}$ is an abelian category since it is the union of $\mathcal{A}_a$ for $a\geq 0$.
\end{proof}
\begin{prop}\label{5.4}
  Assume $({\rm C}_{\leq n})$, $({\rm Mur}_{\leq n})$, $({\rm Semi}_{\leq n})$, and $({\rm wVan}_{\leq n})$. Let $f:M\rightarrow N$ be a morphism in $W^d{\rm MM}_{\leq n}$.
  \begin{enumerate}[{\rm (1)}]
    \item If $M$ is a pure $n$-motive of weight $d$, then ${\rm ker}\,f$ is a pure $n$-motive of weight $d$.
    \item If $N$ is a pure $n$-motive of weight $d$, then ${\rm cok}\,f$ is a pure $n$-motive of weight $d$.
  \end{enumerate}
\end{prop}
\begin{proof}
  (1) In the proof of Theorem \ref{5.2}, we have a distinguished triangle
  \[{\rm ker}\,Gr_df\rightarrow {\rm ker}\,f\rightarrow {\rm ker}\,W^{d-1}f\rightarrow {\rm ker}\,Gr_df[1]\]
  in $\dmeff$. Since ${\rm ker}\,W^{d-1}f=0$ by assumption, ${\rm ker}\,Gr_df\cong {\rm ker}\,f$. This is a pure $n$-motive of weight $d$.\\[5pt]
  (2) Its proof is dual to that of (1).
\end{proof}
\section{Motives associated with schemes}
\begin{df}\label{6.6}
  Assume $({\rm C}_{\leq n})$. Let $\mathcal{C}_{\leq n}$ denote the category defined as follows.
  \begin{enumerate}[(i)]
    \item An object $F$ is a sequence
    \[M_{\leq 2n}(F)\rightarrow M_{\leq 2n-1}(F)\rightarrow \cdots \rightarrow M_{\leq -1}(F)=0\]
    in $\dmeff$ such that for each $d$, $M_d(F)[-d]$ is an $n$-motive of weights $\leq d$ for each $d$. Here, let $M_d(F)$ be a cocone of the morphism $M_{\leq d}(F)\rightarrow M_{\leq d-1}(F)$ in $\dmeff$.
    \item A morphism $F\rightarrow G$ is a commutative diagram
      \[\begin{tikzcd}
    M_{\leq 2n}(F)\arrow[d,"f_{\leq 2n}"]\arrow[r]&M_{\leq 2n-1}(F)\arrow[d,"f_{\leq 2n-1}"]\arrow[r]&\cdots\arrow[r]&M_{\leq 0}(F)\arrow[d,"f_{\leq 0}"]\\
    M_{\leq 2n}(G)\arrow[r]&M_{\leq 2n-1}(G)\arrow[r]&\cdots\arrow[r,"b_1"]&M_{\leq 0}(G)
  \end{tikzcd}\]
    \item The composition of morphisms is given by composing the commutative diagrams.
  \end{enumerate}
  We have the functor $\pi:\mathcal{C}_{\leq n}\rightarrow \dmeff$ given by
  \[F\mapsto M_{\leq 2n}(F).\]
  We often omit $\pi$ for brevity. For each $d$, we denote by $M_{\leq d}\mathcal{C}_{\leq n}$ (resp.\ $M_{\geq d}\mathcal{C}_{\leq n}$ the full subcategory of $\mathcal{C}_{\leq n}$ consisting of objects $F$ such that $M_{\leq e}(F)=0$ for $e>d$ (resp.\ $e<d$).
\end{df}
\begin{lemma}\label{6.7}
  Assume $({\rm C}_{\leq n})$, $({\rm Mur}_{\leq n})$, and $({\rm wVan}_{\leq n}')$. Let $F$ be an object of $M_{\geq d}\mathcal{C}_{\leq n}$, and $G$ be objects of $M_{\leq e}\mathcal{C}_{\leq n}$. If $d>e$, then
  \[{\rm Hom}(F,G)=0,\quad {\rm Hom}(F,G[-1])=0.\]
\end{lemma}
\begin{proof}
  Let us use an induction on $d$. If $e=-1$, then we are done since $\pi(F)=0$. Hence assume that $e\geq 0$. Consider the distinguished triangle
  \[M_{e}(G)\rightarrow M_{\leq e}(G)\rightarrow M_{\leq e-1}(G)\rightarrow M_e(G)[1]\]
  in $\dmeff$. By induction,
  \[{\rm Hom}(F,M_{\leq e-1}(G))=0,\quad {\rm Hom}(F,M_{\leq e-1}(G)[-1])=0.\]
  Thus we reduce to the case when $G=M_e(G)$. Similarly, we reduce to the case when $F=M_d(F)$.

  Then ${\rm Hom}(F,G)=0$ follows from $({\rm Mur}_{\leq n})$, and ${\rm Hom}(F,G[-1])=0$ follows from $({\rm wVan}_{\leq n}')$.
\end{proof}
\begin{prop}\label{6.8}
  Assume $({\rm C}_{\leq n})$, $({\rm Mur}_{\leq n})$, and $({\rm wVan}_{\leq n}')$. Then the functor $\pi:\mathcal{C}_{\leq n}\rightarrow \dmeff$ is fully faithful.
\end{prop}
\begin{proof}
  Let $f:\pi(F)\rightarrow \pi(G)$ be a morphism in $\dmeff$. Set $f_{\leq 2n}:=f$, and consider the commutative diagram
  \[\begin{tikzcd}
    M_{2n}(F)\arrow[r]&M_{\leq 2n}(F)\arrow[d,"f_{\leq 2n}"]\arrow[r]&M_{\leq 2n-1}(F)\arrow[r]&M_{2n}(F)[1]\\
    M_{2n}(G)\arrow[r]&M_{\leq 2n}(G)\arrow[r]&M_{\leq 2n-1}(G)\arrow[r]&M_{2n-1}(G)[1]
  \end{tikzcd}\]
  in $\dmeff$. By \ref{6.7},
  \[{\rm Hom}(M_{2n}(F),M_{\leq 2n-1}(G))=0,\quad {\rm Hom}(M_{2n}(F),M_{\leq 2n-1}(G)[-1])=0.\]
  Thus by Lemma \ref{3.2}, there is a unique morphism $f_{\leq 2n-1}:M_{\leq 2n-1}(F)\rightarrow M_{\leq 2n-1}(G)$ in $\dmeff$ making the above diagram commutative. Repeating this process, we can construct uniquely a morphism $g:F\rightarrow G$ in $\mathcal{C}_n$ such that $\pi(g)=f$.
\end{proof}
\begin{prop}\label{6.10}
  Assume $({\rm C}_{\leq n})$, $({\rm Mur}_{\leq n})$, and $({\rm wVan}_{\leq n}')$.
  \begin{enumerate}[{\rm (1)}]
    \item For any $d$, the inclusion functor $M_{\leq d}\mathcal{C}_{\leq n}\rightarrow \mathcal{C}_{\leq n}$ admits a left adjoint denoted by $M_{\leq d}$.
    \item For any $d$, the inclusion functor $M_{\geq d}\mathcal{C}_{\leq n}\rightarrow \mathcal{C}_{\leq n}$ admits a right adjoint denoted by $M_{\geq d}$.
    \item For any $d$, there is a natural transformation $\delta:M_{\leq d-1}\rightarrow M_{\geq d}[1]$ such that
    \[M_{\geq d}\rightarrow {\rm id}\rightarrow M_{\leq d-1}\stackrel{\delta}\rightarrow M_{\geq d}[1]\]
    is a distinguished triangle of functors.
    \item For any $d>e$, there is a natural isomorphism $M_{\leq d}M_{\geq e}\stackrel{\sim}\rightarrow M_{\geq e}M_{\leq d}$ of functors.
    \item For any object $F$ of $M_{\leq d}\mathcal{C}_{\leq n}$, $\underline{\rm Hom}(F,{\bf L}^n)$ is an object of $M_{\geq 2n-d}\mathcal{C}_{\geq n}$.
  \end{enumerate}
\end{prop}
\begin{proof}
  The proofs of (1), (2), (3), and (4) are parallel to those of Theorem \ref{3.7}, Propositions \ref{3.11}, \ref{3.12}, and \ref{3.13} respectively if we use Lemma \ref{6.7} instead of Proposition \ref{3.6}. The proof of (5) is parallel to that of Theorem \ref{4.2} if we use Lemma \ref{6.7} and Theorem \ref{4.2} instead of Proposition \ref{3.6} and the fact that $\underline{\rm Hom}(M,{\bf 1}(n))$ is a pure $n$-motive of weights $2n-d$ for any pure $n$-motive $M$ of weight $d$.
\end{proof}
\begin{df}\label{6.13}
  Assume $({\rm C}_{\leq n})$, $({\rm Mur}_{\leq n})$, and $({\rm wVan}_{\leq n}')$. For each $d$, set
  \[M_d:=M_{\geq d}M_{\leq d}.\]
\end{df}
\begin{prop}\label{6.11}
  Assume $({\rm C}_{\leq n})$, $({\rm Mur}_{\leq n})$, and $({\rm wVan}_{\leq n}')$. Let $f:F\rightarrow G$ be a morphism in $\mathcal{C}_{\leq n}$. Then for each $d$, there is a unique morphism $f_{\leq d}:M_{\leq d}(F)\rightarrow M_{\leq d}(G)$ such that the diagram
  \[\begin{tikzcd}
    F\arrow[d,"f"']\arrow[r]&M_{\leq d}(F)\arrow[d,"f_{\leq d}"]\\
    G\arrow[r]&M_{\leq d}(G)
  \end{tikzcd}\]
  in $\dmeff$ commutes.
\end{prop}
\begin{proof}
  By Proposition \ref{6.10}, we have the diagram
  \[\begin{tikzcd}
    M_{\geq d+1}(F)\arrow[r]&F\arrow[d,"f"]\arrow[r]&M_{\leq d}(F)\arrow[r]&M_{\geq d+1}(F)[1]\\
    M_{\geq d+1}(G)\arrow[r]&G\arrow[r]&M_{\leq d}(G)\arrow[r]&M_{\geq d+1}(G)[1]
  \end{tikzcd}\]
  in $\dmeff$ where the rows are distinguished triangles. By Lemma \ref{6.7},
  \[{\rm Hom}(M_{\geq d+1}(F),M_{\leq d}(G))=0,\quad {\rm Hom}(M_{\geq d+1}(F),M_{\leq d}(G)[-1])=0.\]
  Thus by Lemma \ref{3.2}, there is a unique morphism $f_{\leq d}:M_{\leq d}(F)\rightarrow M_{\leq d}(G)$ making the above diagram commutative.
\end{proof}
\begin{none}\label{6.3}
   Assume $({\rm C}_{\leq n})$, $({\rm Mur}_{\leq n})$, and $({\rm wVan}_{\leq n}')$. Let
  \[F\stackrel{f}\rightarrow G\stackrel{g}\rightarrow H\stackrel{h}\rightarrow F[1]\]
  be a distinguished triangle in $\dmeff$. Assume that $F$ and $G$ be objects of $\mathcal{C}_{\leq n}$. Then by Proposition \ref{6.8}, $f$ can be considered as a morphism in $\mathcal{C}_{\leq n}$. Thus we have the commutative diagram
  \[\begin{tikzcd}
    F=M_{\leq 2n}(F)\arrow[d,"f_{\leq 2n}"']\arrow[r,"a_{2n}"]&M_{\leq 2n-1}(F)\arrow[d,"f_{\leq 2n-1}"]\arrow[r,"a_{2n-1}"]&\cdots\arrow[r,"a_1"]&M_{\leq 0}(F)\arrow[d,"f_{\leq 0}"]\\
    G=M_{\leq 2n}(G)\arrow[r,"b_{2n}"]&M_{\leq 2n-1}(G)\arrow[r,"b_{2n-1}"]&\cdots\arrow[r,"b_1"]&M_{\leq 0}(G)
  \end{tikzcd}\]
  in $\dmeff$.

  Choose distinguished triangles
  \[M_d(F)\stackrel{p_d}\rightarrow M_{\leq d}(F)\stackrel{a_d}\rightarrow M_{\leq d-1}(F)\stackrel{p_d'}\rightarrow M_d(F)[1],\]
  \[M_d(G)\stackrel{q_d}\rightarrow M_{\leq d}(G)\stackrel{b_d}\rightarrow M_{\leq d-1}(G)\stackrel{q_d'}\rightarrow M_d(G)[1]\]
  in $\dmeff$. By Propositions \ref{3.9}, \ref{6.7}, and \ref{3.2}, there is a unique morphism $f_d:M_d(F)\rightarrow M_d(G)$ in $\dmeff$ such that the diagram
  \[\begin{tikzcd}
    M_d(F)\arrow[d,"f_d"']\arrow[r,"p_d"]&M_{\leq d}(F)\arrow[d,"f_{\leq d}"]\arrow[r,"a_d"]&M_{\leq d-1}(F)\arrow[d,"f_{\leq d-1}"]\arrow[r,"p_d'"]&M_d(F)[1]\arrow[d,"{f_d[1]}"]\\
    M_d(G)\arrow[r,"q_d"]&M_{\leq d}(G)\arrow[r,"b_d"]&M_{\leq d-1}(G)\arrow[r,"q_d'"]&M_d(G)[1]
  \end{tikzcd}\]
  in $\dmeff$ commutes. Assume that $f_d[-d]$ has a kernel, cokernel, and image in ${\rm MM}_{\leq n}$ for any $d$. By the octahedral axiom, for some objects $U_d$ and $V_d$ of $\dmeff$, there are commutative diagrams
  \[\begin{tikzcd}
    {\rm im}\,f_d[-1]\arrow[r]\arrow[d]&0\arrow[d]\arrow[r]&{\rm im}\,f_d\arrow[d,"\alpha_d"]\arrow[r,"{\rm id}"]&{\rm im}\,f_d\arrow[d]\\
    {\rm ker}\,f_d\arrow[d]\arrow[r,"u_d"]&M_{\leq d}(F)\arrow[d,"{\rm id}"]\arrow[r,"u_d'"]&U_d\arrow[d,"\alpha_d'"]\arrow[r,"u_d''"]&{\rm ker}\,f_d[1]\arrow[d]\\
    M_d(F)\arrow[r,"p_d"]\arrow[d]&M_{\leq d}(F)\arrow[r,"a_d"]\arrow[d]&M_{\leq d-1}(F)\arrow[d,"\alpha_d''"]\arrow[r,"p_d'"]&M_d(F)[1]\arrow[d]\\
    {\rm im}\,f_d\arrow[r]&0\arrow[r]&{\rm im}\,f_d[1]\arrow[r,"{\rm id}"]&{\rm im}\,f_d[1]
  \end{tikzcd}\]
    \[\begin{tikzcd}
    {\rm cok}\,f_d[-1]\arrow[r]\arrow[d]&0\arrow[d]\arrow[r]&{\rm cok}\,f_d\arrow[d,"\beta_d"]\arrow[r,"{\rm id}"]&{\rm cok}\,f_d\arrow[d]\\
    {\rm im}\,f_d\arrow[d]\arrow[r,"v_d"]&M_{\leq d}(G)\arrow[d,"{\rm id}"]\arrow[r,"v_d'"]&V_d\arrow[d,"\beta_d'"]\arrow[r,"v_d''"]&{\rm im}\,f_d[1]\arrow[d]\\
    M_d(G)\arrow[r,"q_d"]\arrow[d]&M_{\leq d}(G)\arrow[r,"b_d"]\arrow[d]&M_{\leq d-1}(G)\arrow[d,"\beta_d''"]\arrow[r,"q_d'"]&M_d(G)[1]\arrow[d]\\
    {\rm cok}\,f_d\arrow[r]&0\arrow[r]&{\rm cok}\,f_d[1]\arrow[r,"{\rm id}"]&{\rm cok}\,f_d[1]
  \end{tikzcd}\]
  in $\dmeff$ where each row and column is a distinguished triangle. Here, we set
  \[{\rm ker}\,f_d:=({\rm ker}(f_d[-d]))[d],\quad {\rm im}\,f_d:=({\rm im}(f_d[-d]))[d],\]
  \[{\rm cok}\,f_d:=({\rm cok}(f_d[-d]))[d]\]
  for brevity.

  Consider the diagram
  \[\begin{tikzcd}
  {\rm ker}\,f_d\arrow[r,"u_d"]& M_{\leq d}(F)\arrow[d,"f_{\leq d}"]\arrow[r,"u_d'"]&U_d\arrow[r,"u_d''"]&{\rm ker}\,f_d[1]\\
     &M_{\leq d}(G)
  \end{tikzcd}\]
  in $\dmeff$. Since $f_{\leq d}u_d=0$, there is a morphism $\eta:U_d\rightarrow M_{\leq d}(G)$ in $\dmeff$ making the above diagram commutative. Then by the octahedral axiom, for some object $M_{\leq d}(H)$ of $\dmeff$, there is a commutative diagram
  \[\begin{tikzcd}
    {\rm im}\,f_d\arrow[r,"{\rm id}"]\arrow[d,"\alpha_d"']&{\rm im}\,f_d\arrow[d,"v_d"]\arrow[r]&0\arrow[d]\arrow[r]&{\rm im}\,f_d[1]\arrow[d,"{\alpha_d[1]}"]\\
    U_d\arrow[d,"\alpha_d'"']\arrow[r,"\eta"]&M_{\leq d}(G)\arrow[d,"v_d'"]\arrow[r,"g_{\leq d}"]&M_{\leq d}(H)\arrow[d,"{\rm id}"]\arrow[r,"\eta''"]&U_d[1]\arrow[d,"{\alpha_d'[1]}"]\\
    M_{\leq d-1}(F)\arrow[r,"\varphi"]\arrow[d,"\alpha_d''"']&V_d\arrow[d,"v_d''"]\arrow[r,"\varphi'"]&M_{\leq d}(H)\arrow[d]\arrow[r,"\varphi''"]&M_{\leq d-1}(F)[1]\arrow[d,"{-\alpha_d''[1]}"]\\
    {\rm im}\,f_d[1]\arrow[r,"{\rm id}"]&{\rm im}\,f_d[1]\arrow[r]&0\arrow[r]&{\rm im}\,f_d[1]
  \end{tikzcd}\]
  in $\dmeff$ such that each row and column is a distinguished triangle. Now choose a distinguished triangle
  \[M_{\leq d-1}(F)\stackrel{f_{\leq d-1}}\rightarrow M_{\leq d-1}(G)\stackrel{\theta_d}\rightarrow W_d\stackrel{\theta_d'}\rightarrow M_{\leq d-1}(F)[1]\]
  in $\dmeff$. By the octahedral axiom, there are commutative diagrams
  \[\begin{tikzcd}
    0\arrow[r]\arrow[d]&{\rm cok}\,f_d\arrow[r,"{\rm id}"]\arrow[d,"\beta_d"]&{\rm cok}\,f_d\arrow[d,"w_d"]\arrow[r]&0\arrow[d]\\
    M_{\leq d-1}(F)\arrow[r,"\varphi"]\arrow[d,"{\rm id}"']&V_d\arrow[r,"\varphi'"]\arrow[d,"\beta_d'"]&M_{\leq d}(H)\arrow[d,"w_d'"]\arrow[r,"\varphi''"]&M_{\leq d-1}(F)[1]\arrow[d,"{\rm id}"]\\
    M_{\leq d-1}(F)\arrow[d]\arrow[r,"f_{\leq d-1}"]&M_{\leq d-1}(G)\arrow[r,"\theta_d"]\arrow[d,"\beta_d''"]&W_d\arrow[d,"w_d''"]\arrow[r,"\theta_d'"]&M_{\leq d-1}(F)[1]\arrow[d]\\
    0\arrow[r]&{\rm cok}\,f_d[1]\arrow[r,"{\rm id}"]&{\rm cok}\,f_d[1]\arrow[r]&0
  \end{tikzcd}\]
  \[\begin{tikzcd}
    {\rm ker}\,f_d\arrow[r]\arrow[d,"u_d"']&0\arrow[d]\arrow[r]&{\rm ker}\,f_d[1]\arrow[d,"\gamma_{d+1}"]\arrow[r,"{\rm id}"]&{\rm ker}\,f_d[1]\arrow[d,"{u_d[1]}"]\\
    M_{\leq d}(F)\arrow[r,"f_{\leq d}"]\arrow[d,"u_d'"']&M_{\leq d}(G)\arrow[r,"\theta_{d+1}"]\arrow[d,"{\rm id}"]&W_{d+1}\arrow[d,"\gamma_{d+1}'"]\arrow[r,"\theta_{d+1}'"]&M_{\leq d}(F)[1]\arrow[d,"{u_d'[1]}"]\\
    U_d\arrow[r,"\eta"]\arrow[d,"u_d''"']&M_{\leq d}(G)\arrow[d]\arrow[r,"g_{\leq d}"]&M_{\leq d}(H)\arrow[d,"\gamma_{d+1}''"]\arrow[r,"\eta''"]&U_d[1]\arrow[d,"{-u_d''[1]}"]\\
    {\rm ker}\,f_d[1]\arrow[r]&0\arrow[r]&{\rm ker}\,f_d[2]\arrow[r,"-{\rm id}"]&{\rm ker}\,f_d[2]
  \end{tikzcd}\]
  in $\dmeff$ such that each row and column is a distinguished triangle.

  Now by the octahedral axiom, we can choose an object $M_d(H)$ and a commutative diagram
  \[\begin{tikzcd}
    {\rm ker}\,f_{d-1}\arrow[r]\arrow[d,"{-\mu_d''}"']&0\arrow[d]\arrow[r]&{\rm ker}\,f_{d-1}[1]\arrow[r,"{\rm id}"]\arrow[d,"\gamma_d"]&{\rm ker}\,f_{d-1}[1]\arrow[d,"{\mu_d''[1]}"]\\
    {\rm cok}\,f_d\arrow[d,"\mu_d"']\arrow[r,"w_d"]&M_{\leq d}(H)\arrow[r,"w_d'"]\arrow[d,"{\rm id}"]&W_d\arrow[d,"\gamma_d'"]\arrow[r,"w_d''"]&{\rm cok}\,f_d[1]\arrow[d,"{\mu_d[1]}"]\\
    M_d(H)\arrow[d,"\mu_d'"']\arrow[r,"r_d"]&M_{\leq d}(H)\arrow[r,"c_d"]\arrow[d]&M_{\leq d-1}(H)\arrow[d,"\gamma_d''"]\arrow[r,"r_d'"]&M_d(H)[1]\arrow[d,"{\mu_d'[1]}"]\\
    {\rm ker}\,f_{d-1}[1]\arrow[r]&0\arrow[r]&{\rm ker}\,f_{d-1}[2]\arrow[r,"{\rm id}"]&{\rm ker}\,f_{d-1}[2]
      \end{tikzcd}\]
    in $\dmeff$ where each row and column is a distinguished triangle. Then we get the distinguished triangle
    \begin{equation}\label{6.11.1}
    {\rm cok}\,f_d[-d]\stackrel{\mu_d[-d]}\rightarrow M_d(H)[-d]\stackrel{\mu_d'[-d]}\rightarrow {\rm ker}\,f_{d-1}[-d+1]\stackrel{\mu_d''[-d]}\rightarrow {\rm cok}\,f_{d-1}[-d+1]
    \end{equation}
    in $\dmeff$.
\end{none}
\begin{prop}\label{6.9}
   Assume $({\rm C}_{\leq n})$, $({\rm Mur}_{\leq n})$, and $({\rm wVan}_{\leq n}')$. Let $F\stackrel{f}\rightarrow G\stackrel{g}\rightarrow H\rightarrow F[1]$ be a distinguished triangle in $\dmeff$ such that $F$ and $G$ are objects of $\mathcal{C}_{\leq n}$. Assume that
   \begin{enumerate}[{\rm (i)}]
     \item for each $d$, the morphism $f_d[-d]$ in {\rm \ref{6.3}} has a kernel, cokernel, and image in $W^d{\rm MM}_{\leq n}$,
     \item for each $d$, the cokernel of $f_d[-d]$ is a pure $n$-motive of weight $d$.
   \end{enumerate}
   Then $H$ is an object of $\mathcal{C}_{\leq n}$.
\end{prop}
\begin{proof}
  In \ref{6.3}, we have constructed morphisms
  \[H=M_{\leq 2n}(H)\rightarrow M_{\leq 2n-1}(H)\rightarrow\cdots \rightarrow M_{\leq -1}(H)=0\]
  in $\dmeff$ such that for each $d$, $M_d(H)[-d]$ admits the distinguished triangle \eqref{6.11.1}. By assumption, ${\rm cok}\,f_d[-d]$ is a pure $n$-motive of weight $d$. Then $M_d(H)[-d]$ is an $n$-motive of weights $\leq d$ since ${\rm ker}\,f_{d-1}[-d+1]$ is an $n$-motive of weights $\leq d-1$. Thus $H$ is an object of $\mathcal{C}_{\leq n}$.
\end{proof}
\begin{df}\label{6.4}
  Let $\mathscr{T}$ be a triangulated category. We say that a finite sequence
  \[0\rightarrow F_1\stackrel{u_1}\rightarrow \cdots\stackrel{u_{r-1}}\rightarrow F_r\rightarrow 0\]
  in $\mathscr{T}$ with $r\geq 3$ is {\it exact} if there are distinguished triangles
  \[F_1\stackrel{u_1}\rightarrow F_2\stackrel{w_2} \rightarrow C_2\rightarrow F_1[1],\]
  \[C_2\stackrel{v_2}\rightarrow F_2\stackrel{w_3} \rightarrow C_3\rightarrow C_2[1],\]
  \[\cdots,\]
  \[C_{r-2}\stackrel{v_{r-2}}\rightarrow F_{r-2}\stackrel{w_{r-1}} \rightarrow C_{r-1}\rightarrow C_{r-2}[1],\]
  \[C_{r-1}\stackrel{v_{r-1}}\rightarrow F_{r-1}\stackrel{u_{r-1}} \rightarrow F_r\rightarrow C_{r-1}[1]\]
  in $\mathscr{T}$ such that $u_i=v_iw_i$ for $i=2,\ldots,r-2$.
\end{df}
\begin{prop}\label{6.5}
  Under the notations and hypotheses of {\rm \ref{6.3}}, the finite sequence
  \[\begin{split}
    0&\longrightarrow M_{2n+1}(H)[-2n-1]\stackrel{h_{2n+1}[-2n-1]}\longrightarrow M_{2n}(F)[-2n]\stackrel{f_{2n}[-2n]}\longrightarrow M_{2n}(G)[-2n]\stackrel{g_{2n}[-2n]}\longrightarrow M_{2n}(H)[-2n]\\
    &\stackrel{h_{2n}[-2n]}\longrightarrow \cdots   \stackrel{h_0}\longrightarrow M_0(F)\stackrel{f_0}\longrightarrow M_0(G)\longrightarrow M_0(H)  \longrightarrow 0
  \end{split}\]
  is exact in the sense of {\rm Definition \ref{6.4}}.
\end{prop}
\begin{proof}
  This follows from the distinguished triangles
  \[{\rm ker}\,f_d[-d]\rightarrow M_d(F)[-d]\rightarrow {\rm im}\,f_d[-d]\rightarrow {\rm ker}\,f_d[-d+1],\]
  \[{\rm im}\,f_d[-d]\rightarrow M_d(G)[-d]\rightarrow {\rm cok}\,f_d[-d]\rightarrow {\rm im}\,f_d[-d],\]
  \[{\rm cok}\,f_d[-d]\stackrel{\mu_d[-d]}\rightarrow M_d(H)[-d]\stackrel{\mu_d'[-d]}\rightarrow {\rm ker}\,f_{d-1}[-d+1]\stackrel{\mu_d''[-d+1]}\rightarrow {\rm cok}\,f_d[-d+1]\]
  in $\dmeff$.
\end{proof}
\begin{prop}\label{6.14}
  Assume $({\rm CK}_{\leq n})$. Then for any $X\in SmProj_{\leq n}$, $M(X)$ is an object of $\mathcal{C}_{\leq n}$.
\end{prop}
\begin{proof}
  We just need to set $M_{\leq d}(X):=M_0(X)\oplus \cdots \oplus M_d(X)$, which is possible because of $({\rm CK}_{\leq n})$.
\end{proof}
\begin{df}\label{6.1}
  Let $Z$ be a scheme of finite type over $k$, and let $\mathscr{Z}=\{Z_1,\ldots,Z_r\}$ be a closed cover of $Z$.
Let ${\rm C}({\rm PSh}^{tr})$ denote the category of complexes of presheaves with transfers on the category of schemes smooth over $k$.
Recall the object $M(X):=C_*{\bf Z}_{tr}(X)$ of ${\rm C}({\rm PSh}^{tr})$ in \cite{MVW}.
Consider the \v{C}ech double complex
  \begin{equation}\label{6.1.1}
    \bigoplus_{|I|=r}M(Z_I)\rightarrow\cdots\rightarrow \bigoplus_{|I|=1}M(Z_I)
  \end{equation}
  in ${\rm C}({\rm PSh}^{tr})$ where $Z_I:=Z_{i_1}\times_Z \cdots \times_Z Z_{i_r}$ if $I=\{i_1,\ldots,i_r\}$. We denote by $M(\mathscr{Z})$ its total complex.
\end{df}
\begin{lemma}\label{6.2}
  Under the above notations and hypotheses, the induced morphism
  \[M(\mathscr{Z})\rightarrow M(Z)\]
  in $\dmeff$ is an isomorphism.
\end{lemma}
\begin{proof}
  Let us use an induction on $r$. If $r=1$, then this is obvious. If $r=2$, then this follows from \cite[Theorems 16.1.3, 16.1.4]{CD12}. For $r>2$, consider the closed covers
  \[\mathscr{W}:=\{Z_1,\ldots,Z_{r-1}\},\]
  \[\mathscr{W}':=\{Z_{1,r},\ldots,Z_{r-1,r}\}\]
  of $W:=Z_1\cup \cdots \cup Z_{r-1}$ and $W':=Z_{1,r}\cup \cdots \cup Z_{r-1,r}$ respectively. Here, set $Z_{i,r}:=Z_i\times_Z Z_r$. Then $M(\mathscr{W})$ and $M(\mathscr{W'})$ are the total complexes of the double complexes
  \[\bigoplus_{|I|=r,\;r\notin I}M(Z_I)\rightarrow\cdots\rightarrow \bigoplus_{|I|=1\;m\notin I}M(Z_I),\]
  \[\bigoplus_{|I|=r,\;r\in I}M(Z_I)\rightarrow\cdots\rightarrow \bigoplus_{|I|=2\;m\in I}M(Z_I)\]
  in ${\rm C}({\rm PSh}^{tr})$ respectively. Combining with \eqref{6.1.1}, we have the induced distinguished triangle
  \[M(\mathscr{W}')\rightarrow M(\mathscr{W})\oplus M(Z_1)\rightarrow M(\mathscr{Z})\rightarrow M(\mathscr{W}')[1]\]
  in $\dmeff$. Now consider the induced commutative diagram
  \[\begin{tikzcd}
    M(\mathscr{W}')\arrow[d]\arrow[r]& M(\mathscr{W})\oplus M(Z_1)\arrow[d]\arrow[r]& M(\mathscr{Z})\arrow[d]\arrow[r]& M(\mathscr{W}')[1]\arrow[d]\\
    M(W')\arrow[r]& M(W)\oplus M(Z_1)\arrow[r]& M(Z)\arrow[r]& M(W')[1]
  \end{tikzcd}\]
  in $\dmeff$. The claim holds for $r=2$, so the second row is a distinguished triangle. The first, second, and fourth vertical arrows are isomorphisms by induction on $r$. Thus the third vertical arrow is an isomorphism.
\end{proof}
\begin{lemma}\label{6.16}
  Assume $({\rm Res}_{\leq n})$. Let $U$ be an integral scheme smooth and separated over $k$. Then
  \[\underline{\rm Hom}(M^c(U),{\bf L}^n)\cong M(U).\]
\end{lemma}
\begin{proof}
  Let $\mathcal{T}$ be the full subcategory of $\dmeff$ generated by objects of the form $M(X)[n]$ where $n\in {\bf Z}$ and $X$ is a scheme smooth and projective over $k$ whose dimension is $\leq n$. By \cite[Theorem 4.3.2]{Voe2}, the functor
  \[\mathcal{T}\rightarrow \mathcal{T}\]
  given by $F\mapsto \underline{\rm Hom}(\underline{\rm Hom}(F,{\bf L}^n),{\bf L}^n)$ is an equivalence. In the proof of \cite[Theorem 4.3.7]{Voe2}, we have that
  \[\underline{\rm Hom}(M(U),{\bf L}^n)\cong M^c(U).\]
  Thus the conclusion follows from these.
\end{proof}
\begin{none}\label{6.15}
  Assume $({\rm C}_{\leq n})$, $({\rm Mur}_{\leq n})$, $({\rm wVan}_{\leq {n-1}}')$, $({\rm Semi}_{\leq n-1})$, and $({\rm Res}_{\leq n})$, and let $U$ be an integral scheme smooth and separated over $k$ whose dimension is $\leq n$. By $({\rm Res}_{\leq n})$, we can choose a closed immersion $i:Z\rightarrow X$ and a projective smooth morphism $X\rightarrow k$ of schemes such that
  \begin{enumerate}
    \item[(i)] the complement of $i$ is $U$,
    \item[(ii)] $Z=Z_1\cup \cdots \cup Z_r$ is a divisor with strict normal crossings.
  \end{enumerate}
  For $1\leq i\leq r$, let $K_i$ be the total complex of the double complex
  \[\oplus_{|I|=r}M(Z_I)\rightarrow \cdots \rightarrow \oplus_{|I|=i}M(Z_I)\]
  in ${\rm C}({\rm PSh}^{tr})$ where $Z_I:=Z_{i_1}\times_Z \cdots \times_Z Z_{i_r}$ if $I=\{i_1,\ldots,i_r\}$. Then we have the distinguished triangle
  \[K_{i+1}\rightarrow \oplus_{|I|=i+1}M(Z_I)\rightarrow K_i\rightarrow K_{i+1}[1]\]
  for each $1\leq i\leq r-1$. By Proposition \ref{6.14}, $\oplus_{|I|=i}M(Z_I)$ is an object of $\mathcal{C}_{\leq n-2}$ if $i>1$ and of $\mathcal{C}_{\leq n-1}$ if $i=1$. In particular, $K_r$ is an object of $\mathcal{C}_{\leq n-2}$ if $r>1$ and of $\mathcal{C}_{\leq n-1}$ if $r=1$. By Theorem \ref{5.2} and Proposition \ref{5.4}, we can apply Proposition \ref{6.9} to the above distinguished triangle. Thus each $K_i$ is an object of $\mathcal{C}_{\leq n-2}$ if $i>0$ and of $\mathcal{C}_{\leq n-1}$ if $i=0$. in particular, $M(Z)=K_0$ is an object of $\mathcal{C}_{\leq n-1}$. In Proposition \ref{6.5}, we have the sequence
  \[M_{2n-2}(K_1)\rightarrow \oplus_{|I|=1}M_{2n-2}(Z_I)\rightarrow M_{2n-2}(K_0)\rightarrow M_{2n-3}(K_1).\]
  Since $K_1$ is an object of $\mathcal{C}_{\leq n-2}$, $M_{2n-2}(K_1)=M_{2n-3}(K_1)=0$. By Proposition \ref{6.5},
  \[\oplus_{|I|=1}M_{2n-2}(Z_I)\cong M_{2n-2}(K_0).\]
  Thus $M_{2n-2}(Z)$ is a pure $(n-1)$-motive of weight $2n-2$.

  Let us consider the following conjecture, which is a weaker version of $({\rm Semi}_{\leq n})$.
  \begin{enumerate}
    \item[$({\rm Semi}_{\leq n}')$] Assume $({\rm Semi}_{\leq n-1})$. For any $d$ and morphism $f:M\rightarrow N$ in ${\rm MM}_{\leq n}$ such that
    \begin{enumerate}[(i)]
      \item ($d<2n-2$ and $M$ is an $(n-1)$-motive of weights $\leq d$) or ($d=2n-2$ and $M$ is a pure $(n-1)$-motive of weight $d$),
      \item $N$ is a pure $n$-motive of weight $d$,
    \end{enumerate}
    the kernel, cokernel, and image of $f$ exist, and the cokernel of $f$ is a pure $n$-motive of weight $d$.
  \end{enumerate}
  Assume $({\rm Semi}_{\leq n}')$. Then we can apply Proposition \ref{6.9} to the distinguished triangle
  \[M(Z)\rightarrow M(X)\rightarrow M^c(U)\rightarrow M(Z)[1]\]
  in $\dmeff$ (\cite[Theorem 16.15]{MVW}), so $M^c(U)$ is an object of $\mathcal{C}_{\leq n}$. Since $\underline{\rm Hom}(M^c(U),{\bf L}^n)\cong M(U)$ by Lemma \ref{6.16}, $M(U)$ is an object of $\mathcal{C}_{\leq n}$ by Proposition \ref{6.10}. Thus we get the following result.
\end{none}
\begin{thm}\label{6.12}
  Assume $({\rm CK}_{\leq n})$, $({\rm Mur}_{\leq n})$, $({\rm wVan}_{\leq n}')$, $({\rm Res}_{\leq n})$, and $({\rm Semi}_{\leq n}')$. Then for any integral scheme $U$ smooth over $k$ whose dimension is $\leq n$, $M(U)$ and $M^c(U)$ are objects of $\mathcal{C}_{\leq n}$.
\end{thm}
\section{Realizations}
\begin{none}\label{7.1}
Let
\[
R_\ell:\dmeff\rightarrow {\rm D}_c^b(k_{\et},{\bf Q}_\ell)
\]
denote the $\ell$-adic realization functor (\cite[Remark 7.2.25]{MR3477640}), which commutes with the Grothendieck six operations (\cite[Definition A.1.10]{MR3477640}).
Here, ${\rm D}_c^b(k_{\et},{\bf Q}_\ell)$ denotes the bounded derived category of constructible \'etale ${\bf Q}_\ell$-sheaves.
\end{none}

\begin{lemma}
\label{7.8}
Let $X$ be a purely $n$-dimensional scheme smooth over $k$.
Then there are canonical isomorphisms
\[
R_\ell(M(X))\cong R\Gamma_{\et,c}(X_{\overline{k}},{\bf Q}_{\ell})(n)[2n],\quad R_\ell(M^c(X))\cong R\Gamma_{\et}(X_{\overline{k}},{\bf Q}_{\ell})(n)[2n].
\]
\end{lemma}
\begin{proof}
Let $p:X\rightarrow {\rm Spec}\,k$ denote the structure morphism.
By the Poincar\'e duality (\cite[Definition A.1.10(3)]{MR3477640}), there is a canonical isomorphism
\[
M(X)\cong p_!p^*{\bf 1}(n)[2n].
\]
Thus
\[
R_\ell(M(X))\cong p_!p^*{\bf Q}_\ell (n)[2n]\cong R\Gamma_{\et,c}(X_{\overline{k}},{\bf Q}_{\ell})(n)[2n].
\]

By \cite[Theorem 16.24]{MVW} (see \cite{Kel} to remove the assumption of resolution of singularities), for any scheme $T$ smooth over $k$ and $i\in {\bf Z}$ there is a canonical isomorphism
\[
{\rm Hom}(M(T)[i],M^c(X))\cong {\rm Hom}(M(T\times X)[i],{\bf 1}(n)[2n]).
\]
On the other hand, by the Grothendieck six operations formalism there is a canonical isomorphism
\[
{\rm Hom}(M(T)[i],p_*p^*M(X))\cong {\rm Hom}(M(T\times X),{\bf 1}).
\]
Thus there is a canonical isomorphism
\[
M^c(X)\cong p_*p^*{\bf 1}(n)[2n].
\]
Then
\[
R_\ell(M^c(X))\cong p_*p^*{\bf Q}_\ell (n)[2n]\cong R\Gamma_{\et}(X_{\overline{k}},{\bf Q}_{\ell})(n)[2n].
\]
\end{proof}
\begin{none}
\label{7.9}
Let $X$ be a purely $n$-dimensional scheme smooth and projective over $k$.
By \cite[0.3]{MR1265526}, there is an isomorphism
\[
R\Gamma_{\et}(X_{\overline{k}},{\bf Q}_{\ell})\cong \bigoplus_{d=0}^{2n} H_{\et}^d(X_{\overline{k}},{\bf Q}_\ell)[-d].
\]
Thus with the \'etale homology notation
\[
H_d^{\et}(X_{\overline{k}},{\bf Q}_\ell)\cong H_{\et,c}^{2n-d}(X_{\overline{k}},{\bf Q})(n)\cong H_{\et}^{2n-d}(X_{\overline{k}},{\bf Q})(n)
\]
we have an isomorphism
\begin{equation}
\label{7.9.1}
R_\ell(M(X))\cong \bigoplus_{d=0}^{2n}H_d^{\et}(X_{\overline{k}},{\bf Q}_\ell)[d].
\end{equation}
\end{none}

\begin{none}\label{7.10}
Assume $({\rm C}_{\leq n})$.
For $X\in SmProj_{\leq n}$, the projector $p_d:M(X)\rightarrow M(X)$ in \ref{1.15} corresponds to the projector
\[
\bigoplus_{d=0}^{2n}H_d^{\et}(X_{\overline{k}},{\bf Q}_\ell)[d]\rightarrow H_d^{\et}(X_{\overline{k}},{\bf Q}_{\ell})[d]\rightarrow \bigoplus_{d=0}^{2n}H_d^{\et}(X_{\overline{k}},{\bf Q}_\ell)[d].
\]
Thus we have an isomorphism
\[
R_\ell(M_d(X)[-d])\cong H_d^{\et}(X_{\overline{k}},{\bf Q}_\ell).
\]
  Consider the usual $t$-structure on the derived category ${\rm D}_c^b(k_{\et},{\bf Q}_\ell)$ whose heart is ${\rm Rep}_{{\rm Gal}(\overline{k}/k)}({\bf Q}_\ell)$. Then $R_\ell(M_d(X)[-d])$ is in ${\rm Rep}_{{\rm Gal}(\overline{k}/k)}({\bf Q}_\ell)$.

  For each $d$, let $\tau_{\leq d}$ and $\tau_{\geq d}$ denote the homological truncation functors of ${\rm D}_c^b(k_{\et},{\bf Q}_\ell)$. By the following proposition, we obtain a functor
    \[R_\ell:{\rm MM}_{\leq n}\rightarrow {\rm Rep}_{{\rm Gal}(\overline{k}/k)}({\bf Q}_\ell).\]
\end{none}
\begin{prop}\label{7.2}
  Assume $({\rm C}_{\leq n})$. For any $n$-motive $M$, $R_\ell(M)$ is in ${\rm Rep}_{{\rm Gal}(\overline{k}/k)}({\bf Q}_\ell)$.
\end{prop}
\begin{proof}
  We may assume that $M$ is an $n$-motive of weights $\leq d$. If $d=-1$, we are done since $M=0$. Hence assume that $d\geq 0$. Consider the distinguished triangle
  \[Gr_dM\rightarrow M\rightarrow W^{d-1}M\rightarrow Gr_dM[1]\]
  in $\dmeff$. By induction, $R_\ell(W^{d-1}M)$ is in the heart of ${\rm D}_c^b(k_{\et},{\bf Q}_\ell)$, and by \ref{7.10}, $R_\ell(Gr_dM)$ is in the heart. Thus $R_\ell(M)$ is also in the heart.
\end{proof}
\begin{prop}\label{7.4}
  {\rm (1)} For any object $F$ of $M_{\leq d}\mathcal{C}_{\leq n}$, $R_\ell(F)$ is in $\tau_{\geq 0}\tau_{\leq d}{\rm D}_c^b(k_{\et},{\bf Q}_\ell)$.\\[5pt]
  {\rm (2)} For any object $F$ of $M_{\geq d}\mathcal{C}_{\leq n}$, $R_\ell(F)$ is in $\tau_{\leq 2n}\tau_{\geq d}{\rm D}_c^b(k_{\et},{\bf Q}_\ell)$.
\end{prop}
\begin{proof}
  (1) Let us use an induction on $d$. If $d=0$, then we are done by Proposition \ref{7.2}. Thus assume that $d>0$. Consider the distinguished triangle
  \[R_\ell(M_d(F))\rightarrow R_\ell (F)\rightarrow R_\ell (M_{\leq d-1}(F))\rightarrow R_\ell (M_d(F)[1])\]
  in ${\rm D}_c^b(k_{\et},{\bf Q}_\ell)$. By Proposition \ref{7.2}, $R_\ell(M_d(F)[-d])$ is in $\tau_{\geq 0}\tau_{\leq 0}{\rm D}_c^b(k_{\et},{\bf Q}_\ell)$. Thus $R_\ell(M_d(F))$ is in $\tau_{\geq d}\tau_{\leq d}{\rm D}_c^b(k_{\et},{\bf Q}_\ell)$. Since $R_{\ell}(M_{\leq d-1}(F))$ is in $\tau_{\geq 0}\tau_{\leq d-1}{\rm D}_c^b(k_{\et},{\bf Q}_\ell)$ by induction on $d$, $R_{\ell}(F)$ is in $\tau_{\geq 0}\tau_{\leq d-1}{\rm D}_c^b(k_{\et},{\bf Q}_\ell)$.\\[5pt]
  (2) Let us use an induction on $d$. If $d=2n$, then we are done by Proposition \ref{7.2}. Thus assume that $d<2n$. Consider the distinguished triangle
    \[R_\ell(M_{\geq d+1}(F))\rightarrow R_\ell (F)\rightarrow R_\ell (M_{d}(F))\rightarrow R_\ell (M_{\geq d+1}(F)[1])\]
  in ${\rm D}_c^b(k_{\et},{\bf Q}_\ell)$. As above, $R_\ell(M_d(F)[-d])$ is in $\tau_{\geq d}\tau_{\leq d}{\rm D}_c^b(k_{\et},{\bf Q}_\ell)$.  Since $R_{\ell}(M_{\geq d+1}(F))$ is in $\tau_{\geq d+1}\tau_{\leq 2n}{\rm D}_c^b(k_{\et},{\bf Q}_\ell)$ by induction on $d$, $R_{\ell}(F)$ is in $\tau_{\geq d}\tau_{\leq 2n}{\rm D}_c^b(k_{\et},{\bf Q}_\ell)$.
\end{proof}
\begin{prop}\label{7.5}
  Assume $({\rm C}_{\leq n})$, $({\rm Mur}_{\leq n})$, and $({\rm wVan}_{\leq {n-1}}')$, Let $F$ be an object of $\mathcal{C}_{\leq n}$. Then there is a unique isomorphism
  \[R_\ell(M_{\leq d}(F))\rightarrow \tau_{\leq d}R_\ell(F)\]
  making the diagram
  \[\begin{tikzcd}
    R_\ell(M_{\geq d+1}(F))\arrow[r]&R_\ell(F)\arrow[d,"{\rm id}"]\arrow[r]&R_\ell(M_{\leq d}(F))\arrow[r]&R_\ell(M_{\geq d+1}(F))\\
    \tau_{\geq d+1}R_\ell(F)\arrow[r]&R_\ell(F)\arrow[r]&\tau_{\leq d}R_\ell(F)\arrow[r]&\tau_{\geq d+1}R_\ell(F)
  \end{tikzcd}\]
  in $\dmeff$ commutative.
\end{prop}
\begin{proof}
  By Proposition \ref{7.4}, $R_{\ell}(M_{\geq d+1}(F))$ is in $\tau_{\geq d+1}{\rm D}_c^b(k_{\et},{\bf Q}_\ell)$. Thus the conclusion follows from Propositions \ref{3.2} and \ref{3.14}.
\end{proof}
\begin{prop}\label{7.6}
    Assume $({\rm C}_{\leq n})$, $({\rm Mur}_{\leq n})$, and $({\rm wVan}_{\leq {n-1}}')$, Let $F$ be an object of $\mathcal{C}_{\leq n}$. Then there is a unique isomorphism
  \[R_\ell(M_d(F))\rightarrow \tau_{\geq d}\tau_{\leq d}R_\ell(F)\]
  making the diagram
  \[\begin{tikzcd}
    R_\ell(M_d(F))\arrow[r]&R_\ell(M_{\leq d}(F))\arrow[d,"\sim"]\arrow[r]&R_\ell(M_{\leq d-1}(F))\arrow[r]&R_\ell(M_d(F))[1]\\
    \tau_{\geq d}\tau_{\leq d}R_\ell(F)\arrow[r]&\tau_{\leq d}R_\ell(F)\arrow[r]&\tau_{\leq d-1}R_\ell(F)\arrow[r]&\tau_{\geq d}\tau_{\leq d}R_\ell(F)
  \end{tikzcd}\]
  in $\dmeff$ commutative. Here, the vertical arrow is defined in {\rm Proposition \ref{7.5}}.
\end{prop}
\begin{proof}
  By Proposition \ref{7.4}, $R_{\ell}(M_d(F))$ is in $\tau_{\geq d}\tau_{\leq d}{\rm D}_c^b(k_{\et},{\bf Q}_\ell)$.  Thus the conclusion follows from Propositions \ref{3.2} and \ref{3.14}.
\end{proof}
\begin{thm}\label{7.7}
  Assume $({\rm CK}_{\leq n})$, $({\rm Mur}_{\leq n})$, $({\rm wVan}_{\leq n}')$, $({\rm Res}_{\leq n})$, and $({\rm Semi}_{\leq n}')$. Then for any integral scheme $U$ smooth over $k$ whose dimension is $\leq n$,
  \[R_\ell(M_d(U)[-d])\cong H_d^{\et}(U_{\overline{k}},{\bf Q}_{\ell}),\quad R_\ell(M_d^c(U)[-d])\cong H_{d,c}^{\et}(U_{\overline{k}},{\bf Q}_{\ell}).\]
\end{thm}
\begin{proof}
We may assume that $U$ has pure $n$-dimension.
Here, recall that we use the \'etale homology notation
\[
H_d^{\et}(U_{\overline{k}},{\bf Q}_\ell):=H_{\et,c}^{2n-d}(U_{\overline{k}},{\bf Q}_{\ell})(n),\quad H_{d,c}^{\et}(U_{\overline{k}},{\bf Q}_\ell):=H_{\et}^{2n-d}(U_{\overline{k}},{\bf Q}_{\ell})(n).
\]
By  Propositions \ref{6.12} and \ref{7.6}, it suffices to show that
\[
\tau_{\geq d}\tau_{\leq d}R_{\ell}(M(U))
\cong H_d^{\et}(U_{\overline{k}},{\bf Q}_{\ell}),
\quad
\tau_{\geq d}\tau_{\leq d}R_{\ell}(M^c(U))
\cong H_{d,c}^{\et}(U_{\overline{k}},{\bf Q}_{\ell}).\]
These follow from Lemma \ref{7.8}.
\end{proof}
\section{2-motives}
\begin{prop}\label{8.4}
  The condition $({\rm Semi}_{\leq 1})$ holds.
\end{prop}
\begin{proof}
  Let $X$ be a connected curve smooth and projective over $k$ with $k':=\Gamma(X,\mathcal{O}_X)$. Then $M(X)\cong M_0(X)\oplus M_1(X)\oplus M_2(X)$ with
  \[M_0(X)\cong M(k'),\quad M_1(X)\cong {\rm Pic}^0(X),\quad M_2(X)\cong M(k')\otimes {\bf L}.\]
  The rational equivalence and numerical equivalence are equal for codimension 0. Thus by \cite{Jan3}, ${\rm Hom}(M_0(X),M_0(X))$ is a semisimple ring. By \cite{Voe},
  \[{\rm Hom}(M_2(X),M_2(X))\cong {\rm Hom}(M_0(X),M_0(X)).\]
  Thus ${\rm Hom}(M_2(X),M_2(X))$ is a semisimple ring. The category of abelian varieties up to isogeny is a semisimple category, and ${\rm Hom}(M_1(X),M_1(X))$ is isomorphic to the group of homomorphisms ${\rm Pic}^0(X)\rightarrow {\rm Pic}^0(X)$ of abelian varieties up to isogeny by \cite[Proposition 4.5]{Sch}. These mean that ${\rm Hom}(M_1(X),M_1(X))$ is a semisimple ring.
\end{proof}
\begin{prop}\label{8.5}
  The condition $({\rm Van}_{\leq 1})$ holds.
\end{prop}
\begin{proof}
   Since $({\rm Mur}_{\leq 2})$ holds by \cite[Theorem 7.3.10]{KMP}, it remains  to show that for any $X,Y\in SmProj_{\leq 1}$, $p<0$, and $2\geq d,e\geq 0$ with $d>e+p-1$,
  \[{\rm Hom}(M_d(X),M_e(Y)[p])=0.\]
  We may assume that $X$ and $Y$ are connected. If $e=0$, then it suffices to show that \[{\rm Hom}(M(X),{\bf 1}[p])=0.\] This holds by \cite[Vanishing Theorem 19.3]{MVW}.

  If $e=1$, then it suffices to show that ${\rm Hom}(M(X),M_1(Y)[p])=0$. Hence it suffices to show that the induced homomorphism
  \[{\rm Hom}(M(X),M(Y)[-1])\rightarrow {\rm Hom}(M(X),M_0(Y)[-1])\]
  of abelian groups is an isomorphism. By \cite[Theorems 16.24, 19.1]{MVW}, it suffices to show that the induced homomorphism
  \[H^{1}(X\times Y,{\bf Z}(1))\rightarrow H^{1}(X\times {\rm Spec}\,k'',{\bf Z}(1))\]
  of motivic cohomology groups is an isomorphism where $k'':=\Gamma(Y,\mathcal{O}_Y)$. Since $X$ and $Y$ are projective over $k$,
  \[H^{1}(X\times Y,{\bf Z}(1))\cong \Gamma(X\times Y,\mathcal{O}_{X\times Y}^*)\cong k'^*\times k''^*,\]
  \[H^{1}(X\times {\rm Spec}\,k'',{\bf Z}(1))\cong \Gamma(X\times {\rm Spec}\,k'',\mathcal{O}_{X\times {\rm Spec}\,k''}^*)\cong k'^*\times k''^*\]
  by \cite[Corollary 4.2]{MVW} where $k':=\Gamma(X,\mathcal{O}_X)$. This proves the claim when $e=1$.

  If $e=2$ and $p\leq -2$, then we are done by \cite[Vanishing Theorem 19.3]{MVW}. If $e=2$ and $p=-1$, then it suffices to show that the induced homomorphism
  \[{\rm Hom}(M(X),{\bf 1}(1)[1])\rightarrow {\rm Hom}(M_0(X),{\bf 1}(1)[1])\]
  is an isomorphism. By \cite[Theorem 16.24]{MVW},  it suffices to show that the induced homomorphism
  \[H^1(X,{\bf Z}(1))\rightarrow H^1(k',{\bf Z}(1))\]
  of motivic cohomology groups is an isomorphism where $k':=\Gamma(X,\mathcal{O}_X)$. Since $X$ is projective over $k$,
  \[H^1(X,{\bf Z}(1))\cong k'^*,\quad H^1(k',{\bf Z}(1))\cong k'^*\]
  by \cite[Corollary 4.2]{MVW}. This proves the claim when $e=2$ and $p=-1$.
\end{proof}

\begin{prop}\label{8.2}
  The condition $({\rm Semi}_{\leq 2}')$ holds.
\end{prop}
\begin{proof}
  Note that $({\rm Semi}_{\leq 1})$ holds by Proposition \ref{8.4}. Let $f:M\rightarrow N$ be a morphism in ${\rm MM}_{\leq 2}$ such that
  \begin{enumerate}[(i)]
    \item ($d<2$ and $M$ is an $1$-motive of weights $\leq d$) or ($d=2$ and $M$ is a pure $1$-motive of weight $d$),
    \item $N$ is a pure $2$-motive of weight $d$.
  \end{enumerate}
  If $d=0$ or $d=1$, then $N$ is a pure $d$-motive of weight $d$ by \cite{Mur}. Since $({\rm Mur}_{\leq 2})$ (resp.\ $({\rm Semi}_{\leq 1})$, resp.\ $({\rm Van}_{\leq 1})$) holds by \cite[Theorem 7.3.10]{KMP} (resp.\ Proposition \ref{8.4}, resp.\ Proposition \ref{8.5}), the kernel, cokernel, and image of $f$ exist by Theorem \ref{5.2}. The cokernel of $f$ is a pure $d$-motive of weight $d$ by Proposition \ref{5.4}. Thus the remaining case is when $d=2$.

  Hence assume that $d=2$.
Then since $M$ is a pure $1$-motive of weight $2$, $M$ is a finite direct sum of ${\bf 1}(1)$.
By \cite[Lemma 7.4.1]{KMP}, there is a decomposition
  \[N\cong N'\oplus N''\]
  in $\dmeff$ for some pure $2$-motives $N'$ and $N''$ of weight $2$ such that
  \[{\rm Hom}({\bf 1}(1),N'')={\rm Hom}(N'',{\bf 1}(1))=0.\quad N'\cong {\bf 1}(1)^{\oplus r}\]
  for some $r\geq 0$. Consider the induced morphism $f':M\rightarrow N'$. From this, we see that
  \[{\rm ker}\,f\cong {\rm ker}\,f',\quad {\rm cok}\,f\cong {\rm cok}\,f'\oplus N'',\quad {\rm im}\,f\cong {\rm im}\,f'.\]
  Thus the kernel, cokernel, and image of $f$ exist, and the cokernel of $f$ is a pure $2$-motive of weight $2$.
\end{proof}
\begin{prop}\label{8.3}
  The condition $({\rm wVan}_{\leq 2}')$ holds.
\end{prop}
\begin{proof}
 We need to show that for any $X,Y\in SmProj_{\leq 2}$ and $4\geq d>e\geq 0$,
  \[{\rm Hom}(M_d(X),M_e(Y)[-1])=0.\]
  Since $\underline{\rm Hom}(M_r(S),{\bf L}^2)\cong M_{4-r}(S)$ by \cite{Mur}, this is equivalent to showing that
  \[{\rm Hom}(M_{4-e}(Y),M_{4-d}(X)[-1])=0.\]
  If $e\geq 2$, then $4-d\leq 1$. Thus we reduce to the case when $e\leq 1$. Assuming this, the conclusion follows from Proposition \ref{8.5}.
\end{proof}
\begin{none}\label{8.7}
  We have verified $({\rm wVan}_{\leq 2}')$ and $({\rm Semi}_{\leq 2}')$. Thus by \ref{1.11}, we get the following result.
\end{none}
\begin{thm}\label{8.6}
When $n=2$, the conclusions of {\rm Theorems \ref{1.6}, \ref{1.7}, and \ref{1.10}} hold.
\end{thm}

\bibliography{bib}
\bibliographystyle{siam}

\end{document}